%% file: main.tex
\documentclass[12pt,reqno]{amsart}
\usepackage[left=3cm,top=2cm,right=3cm,bottom=2cm]{geometry}
\usepackage{amssymb}
\usepackage{amsbsy}
\usepackage{epsfig}
\usepackage{tikz}
\usepackage[english]{babel}    % para texto em InglÃªs
\usepackage{enumitem}
\usepackage{multicol}
\usepackage{blkarray}
\usepackage{tikz}
\usepackage{mathtools}
\usepackage{amsthm}
\usepackage{float}
\usepackage{url}
\usepackage{algorithm}
\usepackage{algpseudocode}

\usepackage[symbol*]{footmisc}

\algnewcommand{\IIf}[1]{\State\algorithmicif\ #1\ \algorithmicthen}
\algnewcommand{\EndIIf}{\unskip\ \algorithmicend\ \algorithmicif}

\algnewcommand{\FFor}[1]{\State\algorithmicfor\ #1\ \algorithmicdo}
\algnewcommand{\EndFFor}{\unskip\ \algorithmicend\ \algorithmicfor}
\usepackage{graphicx}

\newtheorem{teore}{Theorem}[section]

\newtheorem{defn}[teore]{Definition}
\newtheorem{lemat}[teore]{Lemma}
\newtheorem{coro}[teore]{Corollary}

\theoremstyle{definition}
\newtheorem{Example}[teore]{Example}

\DeclareMathOperator{\diam}{diam}
\DeclareMathOperator*{\Spec}{Spec}
\DeclareMathOperator*{\DSpec}{DSpec}

\begin{document}
\DeclarePairedDelimiter\ceil{\lceil}{\rceil}
\DeclarePairedDelimiter\floor{\lfloor}{\rfloor}
%********************************************************

\title[Minimum number of distinct eigenvalues of a threshold graph]{The minimum number of distinct eigenvalues of a threshold graph is at most $4$}

\author[L. E. Allem]{L. Emilio Allem}
 \address{UFRGS - Universidade Federal do Rio Grande do Sul, 
 Instituto de Matem\'atica e Estatística, Porto Alegre, Brazil}\email{emilio.allem@ufrgs.br}

\author[C. Hoppen]{Carlos Hoppen}
\address{UFRGS - Universidade Federal do Rio Grande do Sul, 
Instituto de Matem\'atica e Estatística, Porto Alegre, Brazil}\email{choppen@ufrgs.br}

\author[J. Lazzarin]{João Lazzarin}
 \address{UFSM - Universidade Federal de Santa Maria, 
 Departamento de Matemática, Santa Maria, Brazil}\email{joao.lazzarin@ufsm.br}

\author[L. S. Sibemberg]{Lucas Siviero Sibemberg}
 \address{UFRGS - Universidade Federal do Rio Grande do Sul, 
 Instituto de Matem\'atica e Estatística, Porto Alegre, Brazil}\email{lucas.siviero@ufrgs.br}

\author[F. C. Tura]{Fernando Colman Tura}
 \address{UFSM - Universidade Federal de Santa Maria, 
 Departamento de Matemática, Santa Maria, Brazil}\email{fernando.tura@ufsm.br}



\maketitle

\begin{abstract}
In this note we show that the minimum number of distinct eigenvalues of a threshold graph is at most $4$. Moreover, given any threshold graph $G$ and any nonzero real number $\lambda$, we explicitly construct a matrix $M$ associated with $G$ such that DSpec$(M)\subseteq\{-\lambda,0,\lambda,2\lambda\}$.
\end{abstract}

%%%%%%%%%%%%%%%%%%%%%%%%%%%%%%%%%%%%%%%%%%%%%%%%%%%%%%%%%%%%%%%%%%%%%%%%%%%%%%%%%%%%%%%%%%%%%%%%%%%%%%%%%%%%%%%%%%%%%%%%%%%%%%%%%%%%%%%%%%%%%

\section{Introduction}
Let $G$ be a simple graph with vertex set $V=\{v_1,\ldots,v_n\}$ and edge set $E$. We associate the following collection $S(G)$ of real $n\times n$ symmetric matrices with $G$:
$$S(G)=\left\{A \in \mathbb{R}^{n \times n} \colon A=A^{T} \mbox{ and } \left(\forall~i\neq j \right) \left(a_{ij}\neq 0 \Longleftrightarrow \{v_i,v_j\}\in E\right)\right\}.$$
In other words, the off-diagonal nonzero entries of a mat rix $A \in S(G)$ are precisely the entries $ij$ for which there is an edge $\{v_i,v_j\}\in E$, while the entries on the main diagonal may be chosen freely. For a square matrix $A$, let DSpec$(A)$ denote the set of distinct eigenvalues of $A$. Moreover, for a graph $G$, we define
$$q(G)=\mbox{min}\{|\mbox{DSpec}(A)|:A\in S(G)\}.$$

Extensive research has been conducted on the parameter $q(G)$ across various classes of graphs. For instance, the class of matrices whose underlying graph is a connected acyclic graph (that is, a tree) has attracted the interest of researchers since the early work of Parter~\cite{Parter60} and Wiener~\cite{WIENER1984}, which was followed by work of Leal Duarte, Johnson and their collaborators, see for example~\cite{ParterWiener03,JOHNSON20027,DUARTE1989173,leal2002minimum}. For the state of the art for trees, we refer to the book by Johnson and Saiago \cite{2018eigenvalues}. In particular, it is known that, for any tree $T$, $q(T) \geq \diam(T)+1$, where $\diam(T)$ is the diameter of the tree\footnote{The diameter of a tree $T$ is the number of edges on a longest path in $T$.}. This inequality may be strict, and an important problem is to characterize the trees for which it holds with equality, which are known as diminimal trees. Recently, Johnson and Wakhare~\cite{johnson2022inverse} showed that all linear trees are diminimal, where a tree is linear if all vertices of degree at least three lie on a common path. Furthermore, Allem et al.~\cite{allem2023diminimal} described infinite families of diminimal trees of any fixed diameter, which include infinitely many trees that are not linear. 

The parameter $q(G)$ has also been studied for other classes of graphs. For instance, the paper of Ahmadi et al.~\cite{F2013} gives general bounds for $q(G)$ in terms of other graph parameters and computes its exact value for some classes of graphs, including bipartite graphs and cycles. % complete bipartite graphs and cycles, for instance. 
The graphs for which $q(G)$ is at least $|V(G)|-1$ are characterized in \cite{barrett2017generalizations}. In \cite{levene2022orthogonal} the authors study this parameter for the join of particular types of graphs. Moreover, Barrett et al.~\cite{barrett2023regular} characterized the regular graphs $G$ with degree at most 4 that satisfy $q(G) = 2$. And, for an in-depth discussion of problems of this type, we refer to the book by Hogben, Lin and Shader~\cite{hogben2022inverse}.

The current paper is concerned with threshold graphs, a class of graphs that was independently introduced by Chvátal and Hammer \cite{CHVATAL1977145} and by Henderson and Zalcstein \cite{henderson1977graph} in 1977. They are a natural class of graphs that has numerous applications in diverse areas which include computer science and psychology \cite{mahadev1995threshold}. With respect to the current problem, Fallat and Mojallal~\cite{fallat2022minimum} found properties of connected threshold graphs $G$ with $q(G)=2$ and gave the general upper bound $q(G)\leq\frac{|V(G)|}{4}+C$, where $C$ is an absolute constant, for all connected threshold graphs. Our main contribution is the following strengthening of this bound.
\begin{teore}\label{main_theorem}
If $G$ is a threshold graph and $\lambda \neq 0$ is a real number, then there is a matrix $M\in S(G)$ such that $\DSpec(M)\subseteq\{-\lambda,0,\lambda,2\lambda\}$. In particular, $q(G) \leq 4$.
\end{teore}
Theorem~\ref{main_theorem} makes progress towards a full understanding of the minimum number of distinct eigenvalues of a given threshold graph. We note that this result is tight, given that threshold graphs with binary sequence given by $0^{k_{1}}10^{k_{2}}1$\footnote{See Section~\ref{sec:rep} for the representation of threshold graphs.} of order $n\geq 7$, with $k_{1}\geq 3$ and $k_{2}\geq 2$ have $q(G)=4$ according to Theorem 6.4 in \cite{fallat2022minimum}.

A nice feature of our proof of Theorem~\ref{main_theorem} is that it is constructive, leading to a simple algorithm that, with $G$ and $\lambda$ as an input, produces a matrix $M$ such with underlying graph $G$ and spectrum contained in $\{-\lambda,0,\lambda,2\lambda\}$. Another important aspect of Theorem~\ref{main_theorem} is that, unlike most other results about $q(G)$, for which $G$ is assumed to be connected, we can also obtain a bound on $q(G)$ for disconnected graphs such that every component is a threshold graph. To understand why this holds, consider arbitrary vertex-disjoint graphs $G_1,\ldots,G_\ell$ and let $k=\max_i q(G_i)$. Suppose that there exists a set $S$, $|S|=k$, with the property that there are matrices $M_1,\ldots,M_\ell$ whose underlying graphs are $G_1,\ldots,G_\ell$, respectively, such that $\Spec(M_i)\subseteq S$ for all $i\in\{1,\ldots,\ell\}$. 
We may define a matrix $M$ with underlying graph $G_1 \cup \cdots \cup G_\ell$ by setting each submatrix corresponding to $V(G_i)$ equal to $M_i$. Clearly, $\Spec(M)=S$, so that $q(G_1 \cup \cdots \cup G_\ell)=k$. As a consequence, Theorem \ref{main_theorem} gives the following. 
\begin{coro}
    \label{union}
    Let $G_1,\ldots,G_\ell$ be vertex-disjoint threshold graphs. Then their union satisfies $q(G_1\cup \cdots \cup G_\ell)\leq 4$.
\end{coro}

The paper is organized as follows. 
In Section \ref{sec:rep}, we describe different representations of threshold graphs. Using the so-called bag representation of threshold graphs, we define uniform weighted threshold graphs in Section \ref{sec:uniform}. The symmetries of these graphs are explored in Section \ref{sec:red} to associate the spectrum of a large weighted uniform threshold graph with the spectra of smaller graphs. In Section \ref{sec:cons}, these reductions lead to a constructive process that shows that the minimum number of distinct eigenvalues of threshold graphs is at most four. An explicit algorithm to compute a matrix with this property for every threshold graph and every nonzero $\lambda$ appears in Section \ref{sec:alg}. We conclude the paper with final remarks and open problems in  Section~\ref{sec:concl}.

%  \begin{itemize}
%         \item $q(G)=2$: For some families of graphs, such as the join of a graph with itself, complete bipartite graphs, and cycles,... [AACFMN 2013]. 
%         \item The graphs for which $q(G)\geq |V(G)|-1$  are characterized... [BFHHLS 2017].
%         \item A Nordhaus-Gaddum conjecture for the minimum number of distinct eigenvalues of a graph $q(G)+q(G^c)\leq|G| +2$... [LO\v{S} 2019]
%         \item $q(G)\in\{2,3\}$ for join of graphs...[LO\v{S} 2020]
%         \item It is shown that the minimum number of edges necessary for
% a connected graph G to have $q(G)=2$ is $2n-4$ if $n$ is even, and $2n-3$ if $n$ is odd... [BFFKNRTH 2022]
%     \item This paper considers the
% problem of determining the regular graphs G that satisfy q(G) = 2... [BFFNRT 2023]
%     \item Every linear tree $T$ is diminimal [JW 2022]
%     \end{itemize}

\section{Representations of Threshold Graphs}
\label{sec:rep}

A threshold graph can be constructed from an isolated vertex through an iterative process which, at each step, either adds a new isolated vertex or a new \emph{dominating vertex}, i.e., a vertex that is adjacent to all previously added vertices. Naturally,
this allows us to represent a threshold graph on $n$ vertices as a binary sequence $(b_{1}b_{2}\cdots b_{n})$ such that $b_1=0$ and such that $b_{i}=0$ if vertex $v_{i}$ is added as an isolated vertex, and $b_{i}=1$ if $v_{i}$ is added as a dominating vertex. This binary sequence can also be viewed as a sequence of consecutive blocks of $0$'s and $1$'s represented by $0^{a_{1}}1^{a_{2}}0^{a_{3}}1^{a_{4}}\cdots b_i^{a_i} \cdots b_{r}^{a_{r}}$ where $a_{i}\geq 1$ for $1\leq i\leq r$, and $b_{i}=0$ if $i$ is odd and $b_{i}=1$ if $i$ is even. Obviously, this graph is connected if and only if $r$ is even, or $r=1$ and $a_1=1$. In Figure \ref{binary} we depict the threshold graph $G$ with binary sequence $b=(001110011)=0^{2}1^{3}0^{2}1^{2}$.

\begin{figure}[H]
\begin{tikzpicture}[scale=0.8]
  [scale=0.65,auto=left,every node/.style={circle}]
  \foreach \i/\w in {1/,2/,3/,4/,5/,6/,7/,8/,9/}{
    \node[draw,circle,fill=black,label={360/9 * (\i - 1)+90}:\i] (\i) at ({360/9 * (\i - 1)+90}:3) {\w};} % :3 is the radius
  \foreach \from in {3,4,5}{
    \foreach \to in {1,2,3,4,5,\from}
      \draw (\from) -- (\to);}
      
       \foreach \from in {8,9}{
       \foreach \to in {1,2,3,4,5,6,7,8,\from}
       \draw (\from) -- (\to)
       ;}
\end{tikzpicture}
       \caption{The threshold graph defined by $b=(001110011)=0^{2}1^{3}0^{2}1^{2}$.}
       \label{binary}
 \end{figure}
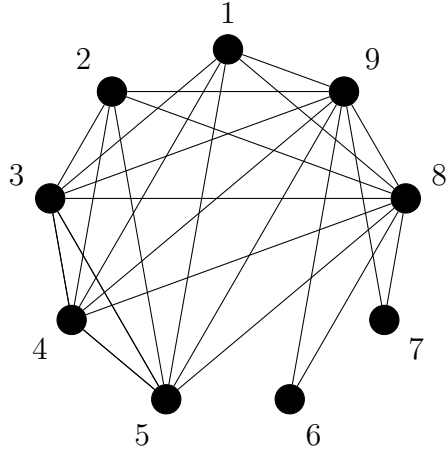

Given that the class of threshold graphs can also be described as the class of graphs that do not contain any graph in $\{P_4,C_4,2K_2\}$ as an induced subgraph, it is a subclass of the well-known class of \emph{complement irreducible graphs} (or \emph{cographs}, for short), which are the graphs that do not contain $P_4$ as an induced subgraph. According to Corneil, Lerchs, and Burlingham~\cite{Corneil1981}, cographs have been independently rediscovered under different names in many disparate areas of mathematics. As is the case for threshold graphs, they can also be constructed in terms of a sequence of operations, as we now explain. Given graphs $G_1=(V_1,E_1)$ and $G_2=(V_2,E_2)$ such that $V_1\cap V_2=\emptyset$, their \emph{union} $G_1 \cup G_2$ is the graph with vertex set $V_1 \cup V_2$ and edge set $E_1 \cup E_2$. Their \emph{join} $G_1 \otimes G_2$ is the graph with vertex set $V_1 \cup V_2$ and edge set $E_1 \cup E_2 \cup \{\{v_1,v_2\} \colon v_1 \in V_1, v_2 \in V_2\}$. A graph $G$ is a cograph if it is an isolated vertex or if $G=G_1 \cup G_2$ or $G=G_1 \otimes G_2$ for two smaller cographs $G_1$ and $G_2$. As a consequence, any cograph can be obtained from a set of isolated vertices by a sequence of operations of type $\cup$ or $\otimes$. 

We shall refer to the names of the vertices corresponding to a block of the binary sequence that defines a threshold graph as the elements of a bag. For instance, in Figure \ref{binary}, the bag corresponding to first block is $B_1=\{1,2\}$, the bag corresponding to the second block is $B_2=\{3,4,5\}$ and so on. 

In this paper, we shall use this representation, which is called the bag representation~\cite{jones2023exploring} of a threshold graph. Precisely, the bag representation of $G$ with binary sequence $(0^{a_{1}}1^{a_{2}}0^{a_{3}}1^{a_{4}}\cdots b_{r}^{a_{r}})$ is a sequence $\{B_{i}\}_{i=1}^{r}$ such that $|B_{i}|=a_{i}$, where a block $0^{a_{i}}$ is associated with a $\cup$-bag $B_{i}$ and a block $1^{a_{i}}$ is associate to a $\otimes$-bag $B_{i}$. Distinct bags are disjoint. For instance, the threshold graph $G$, represented in Figure \ref{binary} has a bag representation with $r=4$ bags, where $B_{1}=\{1,2\}$, $B_{2}=\{3,4,5\}$, $B_{3}=\{6,7\}$ and $B_{4}=\{8,9\}$ as depicted in Figure~\ref{bag_repr}.

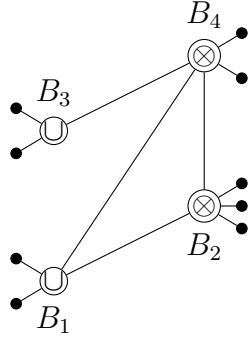
\begin{figure}[H]
\centering
\begin{tikzpicture}
\path(0,0)node[shape=circle,draw,label=below:$B_{1}$,inner sep=0] (bag1) {$\cup$}
    ( -.5,0.3)node[shape=circle,draw,fill=black,inner sep=0pt,minimum size=4pt] (bag12) {}
    ( -.5,-0.3)node[shape=circle,draw,fill=black,inner sep=0pt,minimum size=4pt] (bag11) {}

    (2,1)node[shape=circle,draw,label=below:$B_{2}$,inner sep=0] (bag2) {$\otimes$}
    ( 2.5,1.3)node[shape=circle,draw,fill=black,inner sep=0pt,minimum size=4pt] (bag22) {}
    ( 2.5,1.0)node[shape=circle,draw,fill=black,inner sep=0pt,minimum size=4pt] (bag23) {}
    ( 2.5,0.7)node[shape=circle,draw,fill=black,inner sep=0pt,minimum size=4pt] (bag21) {}

    (0,2)node[shape=circle,draw,label=above:$B_{3}$,inner sep=0] (bag3) {$\cup$}
    ( -0.5,2.3)node[shape=circle,draw,fill=black,inner sep=0pt,minimum size=4pt] (bag32) {}
    ( -0.5,1.7)node[shape=circle,draw,fill=black,inner sep=0pt,minimum size=4pt] (bag31) {}

    (2,3)node[shape=circle,draw,label=above:$B_{4}$,inner sep=0] (bag4) {$\otimes$}
    ( 2.5,3.3)node[shape=circle,draw,fill=black,inner sep=0pt,minimum size=4pt] (bag42) {}
    ( 2.5,2.7)node[shape=circle,draw,fill=black,inner sep=0pt,minimum size=4pt] (bag41) {}
;
    
      \draw[-](bag1)--(bag11);
      \draw[-](bag1)--(bag12);

      \draw[-](bag2)--(bag21);
      \draw[-](bag2)--(bag22);
      \draw[-](bag2)--(bag23);

      \draw[-](bag3)--(bag31);
      \draw[-](bag3)--(bag32);

      \draw[-](bag4)--(bag41);
      \draw[-](bag4)--(bag42);

      \draw[-](bag2)--node[below,sloped] {}(bag1);

      \draw[-](bag4)--node[above,sloped] {}(bag1);
      \draw[-](bag4)--node[right] {}(bag2);
      \draw[-](bag4)--node[above,sloped] {}(bag3);
\end{tikzpicture}
\caption{\label{bag_repr}Bag representation of the threshold graph $0^{2}1^{3}0^{2}1^{2}$.}
\end{figure}

To simplify the notation, we will also accept block representations of the form $(b_j^{a_{j}}b_{j+1}^{a_{j+1}}\cdots b_{r}^{a_{r}})$ with bag representation $\{B_i\}_{i=j}^r$. If $j$ is odd, we have $b_j=0$, i.e., the sequence is a standard bag representation. If $j$ is even, we have $b_j=1$, and we call this representation a $\otimes$-bag representation. In both cases, the sequence $b_j,b_{j+1},\ldots,b_r$ alternates in $\{0,1\}$.

\section{Weighted Threshold Graphs}
\label{sec:uniform}

In this paper, given a threshold graph $G$, instead of working with a matrix $M\in S(G)$ we view it as a weighted threshold graph.

\begin{defn}   A \emph{weighted graph} is a graph with weighted vertices and edges. We write it as a triple $(G, p, \epsilon)$, where $G=(V,E)$ is a simple graph, $p : V \rightarrow \mathbb{R}$ is the function that assigns weights to vertices, and $\epsilon : E  \rightarrow \mathbb{R}\setminus\{0\}$ is the function that assigns weights to edges. The  values of these functions are called \emph{vertex weights} and \emph{edge weights}, respectively.
\end{defn}

The matrix $M = [m_{ij}]$ of a given weighted graph $(G, p, \epsilon)$ with vertices $v_1,v_2, \ldots, v_n$,  is defined by
$$m_{ij} =\left\{
        \begin{array}{ll}
        p_i,  &  \mbox{ if } i = j,\\
        \epsilon_{i,j}, & \mbox{ if } i \not= j \textrm{ and }\{v_i,v_j\} \in E, \\
        0, & \mbox{ if } i \not= j \textrm{ and }\{v_i,v_j\} \notin E.
                 \end{array} \right.
$$

We shall consider a specific type of weighted threshold graph, which we call uniform weighted threshold graph. 
This is a weighted threshold graph $G$ with bag representation $\{B_{i}^{a_{i}}\}_{i=1}^{r}$: 
\begin{itemize}
\item[(a)] each vertex in the bag $B_{i}$ has weight $p_{i}$ for $1\leq i\leq r$;
\item[(b)] the edges connecting the vertices within a
$\otimes$-bag $B_{i}$ have weight $\epsilon_{i}\neq 0$ and for a $\cup$-bag $B_{i}$ we write $\epsilon_{i}=0$ for simplicity;
\item[(c)] each vertex in the $\otimes$-bag $B_{j}$ is connected to all vertices in the bag $B_{i}$ for $i<j$ by an edge with weight $\epsilon_{i,j}\neq 0$. For simplicity, given a $\cup$-bag $B_{j}$, we sometimes write $\epsilon_{i,j}=0$ because there is no edge between the nodes with bags $B_{i}$ and $B_j$.
\end{itemize}
Uniform weighted threshold graphs are denoted $G(p,\epsilon)$. Each such graph corresponds to the symmetric matrix $M(G)=(m_{kl})\in S(G)$ with the following entries:
 \begin{center}
         $m_{kl}=\begin{cases}
			p_{i}, & \mbox{if } k=l \mbox{ and }v_{k}\in B_{i}, \\
            \epsilon_{i} , & \mbox{if } v_{k}, v_{l}\in B_{i}, \textrm{ where }B_{i} \mbox{ is a $\otimes$-bag},\\
            \epsilon_{i,j}, & \mbox{if }i<j \mbox{, }  v_{k}\in B_{i} \mbox{, and }v_{l}\in B_{j}\mbox{, where $B_{j}$ is a $\otimes$-bag},\\
            0, &\mbox{ otherwise.}
		 \end{cases}$
   \end{center}
We will also write $\mbox{Spec}(G(p,\epsilon))$ or simply $\mbox{Spec}(G)$ to refer to $\mbox{Spec}(M)$, where $M\in S(G)$, when it is clear from the context.

\section{Spectral Reduction}
\label{sec:red}

The next result is due to Jones, Trevisan and Vinagre~\cite{jones2023exploring}, and it is a reinterpretation of a more general algorithm presented by Fritscher and Trevisan in \cite{fritscher2016exploring}. Its goal is to reduce a uniform weighted threshold graph $G(p,\epsilon)$ to a simpler uniform weighted threshold graph $H(p',\epsilon')$ such that both have the same spectrum. Even though the version in~\cite{jones2023exploring} was stated for the more general class of cographs, we particularize it to threshold graphs, which are the subject of this paper.

\begin{teore}
\label{thm:red}
Let $G$ be an $n$-vertex uniform weighted threshold graph and bag representation $\left\{ B_i \right\}_{i=1}^r$, with $|B_{i}|=a_{i}$, where for each bag $B_i$, the vertices have weight $p_i$, the internal edges have weight $\epsilon_i$ and external edges have weights $\epsilon_{i,j}$, $1 \leq i<j \leq r$. Then, $G$ and the graph $H$ have the same spectrum, where $H$ is defined as a disjoint union of:
\begin{enumerate}
  \item\label{red:a} A weighted threshold graph $H'$, whose bag representation is $\left\{ B_i' \right\}_{i=1}^r$, with $|B_{i}'|=1$,  with weights $p'$ and $\epsilon'$, where
\begin{equation}\nonumber
    \begin{aligned}
        p'_i & = & p_i+(a_i-1)\epsilon_i,\\
        \epsilon'_{i,j} & = & \sqrt{a_i\cdot a_j}\cdot\epsilon_{i,j}.
    \end{aligned}
\end{equation}
      Since all bags have size 1, there are no weights $\epsilon'_i$.
  \item \label{red:b} A collection of $n-r$ isolated vertices, such that, for each $i\in\{1,\ldots,r\}$, $a_i-1$ have weight $p_i - \epsilon_i$.
\end{enumerate}

\end{teore}

In the next example we apply Theorem \ref{thm:red} to the weighted threshold graph $G=(0^{a_{1}}$$1^{a_{2}}$$0^{a_{3}}$$1^{a_{4}}$$0^{a_{5}}$$1^{a_{6}})$ for some particular vertex and edge weights that will be useful in the remainder of the paper.
% depicted in Figure \ref{fig_1}

\begin{Example}
\label{example_reduction}
Consider the threshold graph $G=(0^{a_{1}}1^{a_{2}}0^{a_{3}}1^{a_{4}}0^{a_{5}}1^{a_{6}})$ for which the weight functions are defined as follows.
The vertex weights in each of the six bags are given by $$p_{1}=0, p_{2}=\frac{\lambda(a_{2}-1)}{a_{2}}, p_{3}=\lambda, p_{4}=\frac{\lambda}{a_{4}}, p_{5}=0 \textrm{ and }p_{6}=\frac{\lambda(a_{6}-1)}{a_{6}}.$$ The weights of the internal edges of the $\otimes$-bags are given by $$\epsilon_{2}=\frac{-\lambda}{a_{2}}, \epsilon_{4}=\frac{\lambda}{a_{4}}, \epsilon_{6}=\frac{-\lambda}{a_{6}}.$$ For the $\cup$-bags, we assume $\epsilon_{1}=0$, $\epsilon_{3}=0$ and $\epsilon_{5}=0$. Finally, for the edge-weights between bags, we have $\epsilon_{1,2}=\frac{\lambda}{\sqrt{a_{1}a_{2}}}$, $\epsilon_{1,4}=\frac{-\lambda}{2}\cdot\frac{1}{\sqrt{a_{1}a_{4}}}$,  $\epsilon_{2,4}=\frac{-\lambda}{2}\cdot\frac{1}{\sqrt{a_{2}a_{4}}}$,  $\epsilon_{3,4}=\frac{-\lambda\sqrt{2}}{2}\cdot\frac{1}{\sqrt{a_{3}a_{4}}}$,  $\epsilon_{1,6}=\frac{\lambda}{4}\cdot\frac{1}{\sqrt{a_{1}a_{6}}}$,  $\epsilon_{2,6}=\frac{\lambda}{4}\cdot\frac{1}{\sqrt{a_{2}a_{6}}}$,  $\epsilon_{3,6}=\frac{\lambda\sqrt{2}}{4}\cdot\frac{1}{\sqrt{a_{3}a_{6}}}$,  $\epsilon_{4,6}=\frac{\lambda}{2}\cdot\frac{1}{\sqrt{a_{4}a_{6}}}$ and  $\epsilon_{5,6}=\frac{\lambda\sqrt{2}}{2}\cdot\frac{1}{\sqrt{a_{5}a_{6}}}$.

We now apply Theorem \ref{thm:red} to the uniform weighted threshold graph $G=G(p,\epsilon)$  to obtain a weighted threshold graph $H$ with the same eigenvalues as $G$.

For $i\in\{1,5\}$ the $a_{i}-1$ isolated vertices have weight $p_{i}-\epsilon_{i}=0$, and the remaining vertex in each $B_i$ produces a bag $B'_{i}$ with weight $p_{i}'=p_{i}+(a_{i}-1)\epsilon_{i}=0$.
For $i\in\{2,6\}$ the $a_{i}-1$ isolated vertices have weight $p_{i}-\epsilon_{i}=\frac{\lambda(a_{i}-1)}{a_{i}}-\frac{-\lambda}{a_{i}}=\lambda$, and the remaining vertex in each $B_{i}$ produces a bag  $B'_{i}$ with weight $p_{i}'=p_{i}+(a_{i}-1)\epsilon_{i}=p_{i}+(a_{i}-1)\epsilon_{i}=\frac{\lambda(a_{i}-1)}{a_{i}}+(a_{i}-1)\frac{-\lambda}{a_{i}}=0$.
For $i=3$ the $a_{3}-1$ isolated vertices have weight $p_{3}-\epsilon_{3}=\lambda$, and the remaining vertex in  $B_3$ produces a bag $B'_{3}$ with weight $p_{3}'=p_{3}+(a_{3}-1)\epsilon_{3}=\lambda$.
Finally, for $i=4$ the $a_{4}-1$ isolated vertices have weight $p_{4}-\epsilon_{4}=\frac{\lambda}{a_{4}}-\frac{\lambda}{a_{4}}=0$, and the remaining vertex in  $B_4$ produces a bag $B'_{4}$ with weight $p_{4}'=p_{4}+(a_{4}-1)\epsilon_{4}=\frac{\lambda}{a_{4}}+(a_{4}-1)\frac{\lambda}{a_{4}}=\lambda$.
%%%%%%%%%%%%%%%%%%%%%%%%%%%%%%

And the new weights of the edges between the bags are given by $\epsilon_{i,j}'=\sqrt{a_{i}a_{j}}\epsilon_{i,j}$ for $1\leq i<j\leq 6$, which leads to 
$\epsilon_{1,2}'=\lambda$, $\epsilon_{1,4}'=\frac{-\lambda}{2}$,  $\epsilon_{2,4}'=\frac{-\lambda}{2}$,  $\epsilon_{3,4}'=\frac{-\lambda\sqrt{2}}{2}$,  $\epsilon_{1,6}'=\frac{\lambda}{4}$,  $\epsilon_{2,6}'=\frac{\lambda}{4}$,  $\epsilon_{3,6}'=\frac{\lambda\sqrt{2}}{4}$,  $\epsilon_{4,6}'=\frac{\lambda}{2}$ and  $\epsilon_{5,6}'=\frac{\lambda\sqrt{2}}{2}$.

Therefore, after applying Theorem \ref{thm:red}, one obtains a weighted threshold graph that is a disjoint union of:
\begin{enumerate}
    \item  a weighted threshold graph $H'$ with binary sequence $(010101)$. Each bag $B'_i$ 
has a single vertex with weight $p'_{i}$ for $1\leq i\leq 6$. The edge weights between bags of $H'$ are given by $\epsilon'_{i,j}$.

     \item $a_{1}+a_{4}+a_{5}-3$ isolated vertices of weight $0$ and $a_{2}+a_{3}+a_{6}-3$ isolated vertices of weight $\lambda$.
\end{enumerate}
Since Theorem \ref{thm:red} ensures that the eigenvalues of $G$ and $H'$ are the same, we have
$$\mbox{Spec}(G)=\mbox{Spec}(H')\cup\{0^{(a_{1}+a_{4}+a_{5}-3)},\lambda^{(a_{2}+a_{3}+a_{6}-3)}\}.$$

\end{Example}

Next we apply Theorem \ref{thm:red} to a uniform weighted threshold graph $G(p,\epsilon)$ that generalizes the graph in Example~\ref{example_reduction}. The choice of the weights produces a reduced graph for which all isolated vertices are assigned one of two distinct weights.

\begin{coro}
    \label{reduced_graph}
Let $G=G(p,\epsilon)$ be a uniform weighted threshold graph with binary representation given by $b= (0^{a_{1}}1^{a_{2}}0^{a_{3}}1^{a_{4}}\cdots b_{r}^{a_{r}})$ with $a_{i}\geq 1$ for $1\leq i\leq r$. The weights $\epsilon_{i,j}$ are arbitrary, but vertex weights are given by

\begin{equation}\nonumber
    \begin{cases}
         p_{i}=0 \mbox{ and } \epsilon_{i}=0 & \mbox{if } i \equiv 1 \,(\bmod~ 4)\\
         p_{i}=\frac{\lambda(a_{i}-1)}{a_{i}} \mbox{ and } \epsilon_{i}=-\frac{\lambda}{a_{i}} & \mbox{if }  i \equiv 2 \,(\bmod~ 4)\\
        p_{i}=\lambda \mbox{ and } \epsilon_{i}=0 & \mbox{if } i \equiv 3 \,(\bmod~ 4) \\
         p_{i}=\frac{\lambda}{a_{i}} \mbox{ and } \epsilon_{i}=\frac{\lambda}{a_{i}} & \mbox{if } i \equiv 0 \,(\bmod~ 4)
    \end{cases}
\end{equation}
Then $G$ and the weighted graph $H$ have the same spectrum, where $H$ is the union of a uniform weighted graph $T_1$ and isolated vertices. Among the isolated vertices, 
$$\displaystyle{ \sum_{\substack{ 1 \leq i \leq n \\ i \equiv 1 \,(\bmod~ 4) }} (a_{i}-1)+\sum_{\substack{ 1 \leq i \leq n \\ i \equiv 0 \,(\bmod~ 4)  }} (a_{i}-1)}$$
have weight $0$ and
$$\displaystyle{ \sum_{\substack{ 1 \leq i \leq n \\ i \equiv 2 \,(\bmod~ 4)  }} (a_{i}-1)+\sum_{\substack{ 1 \leq i \leq n \\ i \equiv 3 \,(\bmod~ 4)  }} (a_{i}-1)}$$ have weight $\lambda$.
The uniform weighted threshold graph $T_1$ has bag representation $\{B^{(1)}_{i}\}_{i=1}^{r}$, with $|B^{(1)}_{i}|=1$, where the edges between bags of $T_{1}$ have weight $\epsilon^{(1)}_{i,j}=\epsilon_{i,j}\sqrt{a_{i}a_{j}}$. Moreover, each bag $B^{(1)}_{i}$ consists of a single vertex $v^{(1)}_i$ with weight $p^{(1)}_{i}$
% in $T_{1}$ has a leaf $v_{i}^{(1)}$ with weight $p^{(1)}_{i}$ for $1\leq i\leq r$ 
given by:
\begin{equation}
\begin{cases}
    p^{(1)}_{i}=0 \mbox{ if }i \equiv j \,(\bmod~ 4)\mbox{ for }j\in\{1,2\} ;\\
    p^{(1)}_{i}=\lambda\mbox{ if }i \equiv j \,(\bmod~ 4)\mbox{ for }j\in\{0,3\}.
\end{cases}
\end{equation}
\end{coro}

Note that the weighted graph $G$ in Example~\ref{example_reduction} has weights as in the statement of Corollary~\ref{reduced_graph} for $r=6$, so that the weighted graph $H'$ computed in the example is precisely the weighted graph $T_1$ produced by the corollary.
% Please keep in mind the notation for $T_{1}$ as the threshold graph generated by $(0101\ldots01)$ with size $r$ because we will refer back to it several times in the text. 

\begin{proof}
We apply Theorem \ref{thm:red} to the uniform weighted threshold graph $G(p,\epsilon)$ with bag representation $\{B_i\}_{i=1}^r$, where $|B_i|=a_i$. This produces a uniform weighted threshold graph $H$ that is the union of isolated vertices and a weighted graph $T_1$ and that has the same spectrum as $G$. Regarding the isolated vertices,
for each bag $B_i$ we remove $a_{i}-1$ vertices and each one is assigned the following weight: 
\begin{equation*}
p_i-\epsilon_i=\begin{cases}
0- 0=0,
\textrm{ if }i\equiv 1~(\bmod~4),\\ 
\frac{\lambda(a_i-1)}{a_i}- \left(-\frac{\lambda}{a_i}\right)=\lambda, \textrm{ if }i\equiv 2~(\bmod~4),\\  
\lambda- 0=\lambda,
\textrm{ if }i\equiv 3~(\bmod~4),\\
\frac{\lambda}{a_i}- \frac{\lambda}{a_i}=0, \textrm{ if }i\equiv 0~(\bmod~4).
\end{cases}
\end{equation*}
Thus, there are
$$\displaystyle{ \sum_{\substack{ 1 \leq i \leq n \\    i\equiv 1~(\bmod~4)  }} (a_{i}-1)+\sum_{\substack{ 1 \leq i \leq n \\ i\equiv 0~(\bmod~4)  }} (a_{i}-1)}$$
isolated vertices of weight $0$ and
$$\displaystyle{ \sum_{\substack{ 1 \leq i \leq n \\ i\equiv 2~(\bmod~4)  }} (a_{i}-1)+\sum_{\substack{ 1 \leq i \leq n \\ i\equiv 3~(\bmod~4)  }} (a_{i}-1)}$$ isolated vertices of weight $\lambda$.

Consider the remaining weighted threshold graph $T_1$ with bag representation $\{B^{(1)}_{i}\}_{i=1}^{r}$, with $|B^{(1)}_{i}|=1$. Notice that the edges between bags of $T_{1}$ have weight $\epsilon^{(1)}_{i,j}=\epsilon_{i,j}\sqrt{a_{i}a_{j}}$. 

For the vertex $v^{(1)}_{i}$ in $B^{(1)}_{i}$, Theorem \ref{thm:red} establishes that
\begin{equation*}p_i^{(1)}=
p_i+(a_i-1)\epsilon_i=\begin{cases}
0+(a_i-1)\cdot 0=0,
\textrm{ if }i\equiv 1~(\bmod~4),\\ 
\frac{\lambda(a_i-1)}{a_i}+(a_i-1)\cdot \left(-\frac{\lambda}{a_i}\right)=0, \textrm{ if }i\equiv 2~(\bmod~4),\\  
\lambda+(a_i-1)\cdot 0=\lambda,
\textrm{ if }i\equiv 3~(\bmod~4),\\
\frac{\lambda}{a_i}+(a_i-1)\cdot \frac{\lambda}{a_i}=\lambda, \textrm{ if }i\equiv 0~(\bmod~4).
\end{cases}
\end{equation*}
Hence, the weights of the vertices of $T_{1}$ are 
\begin{equation}
\begin{cases}
    p^{(1)}_{i}=0\mbox{ if }i\mbox{ mod $4$}\in\{1,2\} ,\\
    p^{(1)}_{i}=\lambda\mbox{ if }i\mbox{ mod $4$}\in\{0,3\}.
\end{cases}
\end{equation}
\end{proof}

Corollary \ref{reduced_graph} shows that after applying a reduction in a uniform weighted threshold graph $G(p,\epsilon)$ then one will continue to work with a simpler weighted graph denoted by $T_{1}$ in order to compute the spectrum of $G$.

The next result shows that if the weights of a weighted threshold graph $T_{k}$ satisfy some additional constraints, then we can reduce it to a simpler threshold graph with one less bag. 
\begin{lemat}
\label{key}
    Let $T_{k}$ be a weighted threshold graph with bag representation given by $\{B_{i}^{(k)}\}_{i=k}^{r}$, with $|B_{i}^{(k)}|=1$, where $B_{k}^{(k)}$ is a $\otimes$-bag if $k$ is even and a $\cup$-bag if $k$ is odd. If $\epsilon^{(k)}_{k,j}=\epsilon^{(k)}_{k+1,j}$ for all $k+1<j\leq r$ such that $j$ even, and $p^{(k)}_{k}=p^{(k)}_{k+1}$, then
    \begin{equation}\label{eq:spec}
    \Spec(T_{k})=\Spec(T_{k+1})\cup\{p^{(k)}_{k}-\epsilon^{(k)}_{k,k+1}\}.
    \end{equation}
Here $T_{k+1}$ is a weighted threshold graph with bag representation given by $\{B_{i}^{(k+1)}\}_{i=k+1}^{r}$, with $|B_{i}^{(k+1)}|=1$ and weights defined as follows
        \begin{itemize}
            \item[(a)] $p_{i}^{(k+1)}=p_{i}^{(k)}$ for $k+2\leq i\leq r$;
            \item[(b)] $p_{k+1}^{(k+1)}=p_{k}^{(k)}+\epsilon^{(k)}_{k,k+1}$;
            \item[(c)] $\epsilon_{i,j}^{(k+1)}=\epsilon_{i,j}^{(k)}$ for $k+1<i<j\leq r$, $j$ even;
            \item[(d)] $\epsilon_{k+1,j}^{(k+1)}=\sqrt{2}\epsilon_{k+1,j}^{(k)}$ for $k+1<j\leq r$, $j$ even.
        \end{itemize}
\end{lemat}

\begin{proof}
Since $\epsilon^{(k)}_{k,j}=\epsilon^{(k)}_{k+1,j}$ for all $k+1<j\leq r$, $j$ even, and $p^{(k)}_{k}=p^{(k)}_{k+1}$ then we can remove the internal node associated with $B^{(k)}_k$ and relocate its leaf $v^{(k)}_{k}$ to the bag of its parent. This new uniform weighted threshold graph $G=G(p',\epsilon')$ has internal nodes labelled by $k+1,\ldots,r$ where the remaining bags are the same as in $T_k$, except for $B'_{k+1}=\{v_k^{(k)},v_{k+1}^{(k)}\}$. The remaining weights are also the same, except in the case where $k+1$ is even (so that $B'_{k+1}$ has type $\otimes$), when we set $\epsilon_{k+1}'=\epsilon_{k,k+1}^{(k)}$.  This relocation is depicted in Figure \ref{fig_lemma12} with $k$ odd. Note that the relocation does not change the weighted threshold graph, just its bag representation, so that $\mbox{Spec}(T_{k})=\mbox{Spec}(G)$. 

\begin{figure}[H]
%\label{lema_relocating}
\centering
\begin{tikzpicture}[scale=0.8]
%%%%%%%%%%%%%%%%%%%%%%%%%%%%%%%%%%%%%%% left    
\path(0,0)node[shape=circle,draw,label=below:$B^{(k)}_{k}$,inner sep=0] (bag3) {$\cup$}
      ( -1,0)node[shape=circle,draw,fill=black,inner sep=0pt,minimum size=4pt,label=left:$p^{(k)}_{k}$] (bag22) {}
      
    (3,1)node[shape=circle,draw,label=below:$B^{(k)}_{k+1}$,inner sep=0] (bag4) {$\otimes$}
    ( 4,1)node[shape=circle,draw,fill=black,inner sep=0pt,minimum size=4pt,label=right:$p^{(k)}_{k+1}$] (bag42) {}
   
    (0,3)node[shape=circle,draw,label=below:$B^{(k)}_{k+2}$,inner sep=0] (bag5) {$\cup$}
     ( -1,3)node[shape=circle,draw,fill=black,inner sep=0pt,minimum size=4pt,label=left:$p^{(k)}_{k+2}$] (bag52) {}
    
    (3,4)node[shape=circle,draw,label=above:$B^{(k)}_{k+3}$,inner sep=0] (bag6) {$\otimes$}
    ( 4,4)node[shape=circle,draw,fill=black,inner sep=0pt,minimum size=4pt,label=right:$p^{(k)}_{k+3}$] (bag62) {}
%%%%%%%%%%%%%%%%%%%%%%%%%%%%%%%%%%%%%%%%%%%%%right
(10,1)node[shape=circle,draw,label=left:$B^{(k)}_{k+1}$,inner sep=0] (nbag4) {$\otimes$}
    ( 11,1)node[shape=circle,draw,fill=black,inner sep=0pt,minimum size=4pt,label=right:$p^{(k)}_{k+1}$] (nbag42) {}
    ( 10,0)node[shape=circle,draw,fill=black,inner sep=0pt,minimum size=4pt,label=left:$p^{(k)}_{k}$] (nbag43) {}
   
    (8,3)node[shape=circle,draw,label=below:$B^{(k)}_{k+2}$,inner sep=0] (nbag5) {$\cup$}
     ( 7,3)node[shape=circle,draw,fill=black,inner sep=0pt,minimum size=4pt,label=left:$p^{(k)}_{k+2}$] (nbag52) {}
    
    (10,4)node[shape=circle,draw,label=above:$B^{(k)}_{k+3}$,inner sep=0] (nbag6) {$\otimes$}
    ( 11,4)node[shape=circle,draw,fill=black,inner sep=0pt,minimum size=4pt,label=right:$p^{(k)}_{k+3}$] (nbag62) {};

      \draw[->] (5,2) -- (5.5,2);
         
      \draw[-](nbag4)--(nbag42);
       \draw[-](nbag4)--node[right] {$\epsilon^{(k)}_{k,k+1}$}(nbag43);
     
      \draw[-](nbag5)--(nbag52);
    
      \draw[-](nbag6)--(nbag62);

       %\draw[-](bag4)--node[below,sloped] {$\epsilon^{(k)}_{k,k+1}$}(bag3);
    
      %\draw[-,red](nbag6)--node[above,sloped] {$\epsilon^{(k)}_{k,k+3}$}(nbag3);
      
      \draw[-](nbag6)--node[right] {$\epsilon^{(k)}_{k+1,k+3}$}(nbag4);
      
      \draw[-](nbag6)--node[above,sloped] {$\epsilon^{(k)}_{k+2,k+3}$}(nbag5);

       \draw [dotted] (nbag6) -- (8.5,5);
    
%%%%%%%%%%%%%%%%%%%%%%%%%%%%%%%%%%%%%%%%%%%%%%%right

      \draw[-](bag3)--(bag22);

      \draw[-](bag4)--(bag42);
     
      \draw[-](bag5)--(bag52);
    
      \draw[-](bag6)--(bag62);
     %%%%%%%%%%%%%%%%%%%%%%
      \draw[-](bag4)--node[below,sloped] {$\epsilon^{(k)}_{k,k+1}$}(bag3);
     %%%%%%%%%%%%%%%%%%%%%%%%%%%%% 
      \draw[-](bag6)--node[above,sloped] {$\epsilon^{(k)}_{k,k+3}$}(bag3);
      
      \draw[-](bag6)--node[right] {$\epsilon^{(k)}_{k+1,k+3}$}(bag4);
      
      \draw[-](bag6)--node[above,sloped] {$\epsilon^{(k)}_{k+2,k+3}$}(bag5);

       \draw [dotted] (bag6) -- (1.5,5);
%%%%%%%%%%%%%%%%%%%%%%%%%%%%%%%%%%%%%%% Right       

\end{tikzpicture}
\caption{\label{fig_lemma12} Vertex $v^{(k)}_{k}$ relocated to the bag $B^{(k)}_{k+1}$.}
\end{figure}
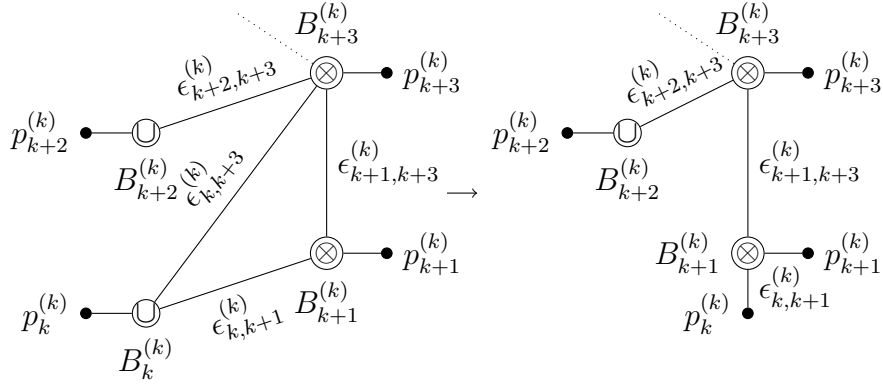

Now we apply Theorem \ref{thm:red} to $G$. Since the only bag with more than one leaf is the bag $B'_{k+1}$ we have that $G$ and the graph $H$ (Figure \ref{fig_lamma22}) have the same spectrum where $H$ is the union of an isolated vertex with weight $p'_k-\epsilon'_k=p^{(k)}_{k}-\epsilon^{(k)}_{k,k+1}$ and a weighted threshold graph $T_{k+1}$ described at the statement.

\begin{figure}[H]
\label{lema_relocating}
\centering
\begin{tikzpicture}[scale=0.8]
%%%%%%%%%%%%%%%%%%%%%%%%%%%%%%%%%%%%%%% left    

%%%%%%%%%%%%%%%%%%%%%%%%%%%%%%%%%%%%%%%%%%%%%right
\path(0,1)node[shape=circle,draw,label=left:$B^{(k)}_{k+1}$,inner sep=0] (nbag4) {$\otimes$}
    ( 1,1)node[shape=circle,draw,fill=black,inner sep=0pt,minimum size=4pt,label=right:$p^{(k)}_{k+1}$] (nbag42) {}
    ( 0,0)node[shape=circle,draw,fill=black,inner sep=0pt,minimum size=4pt,label=left:$p^{(k)}_{k}$] (nbag43) {}
   
    (-2,3)node[shape=circle,draw,label=below:$B^{(k)}_{k+2}$,inner sep=0] (nbag5) {$\cup$}
     ( -3,3)node[shape=circle,draw,fill=black,inner sep=0pt,minimum size=4pt,label=left:$p^{(k)}_{k+2}$] (nbag52) {}
    
    (0,4)node[shape=circle,draw,label=above:$B^{(k)}_{k+3}$,inner sep=0] (nbag6) {$\otimes$}
    ( 1,4)node[shape=circle,draw,fill=black,inner sep=0pt,minimum size=4pt,label=right:$p^{(k)}_{k+3}$] (nbag62) {}

    %%%%%%%%%%%%%%%%%%%%%%%%%%%%% right
(7,1)node[shape=circle,draw,label=left:$B^{(k+1)}_{k+1}$,inner sep=0] (nnbag4) {$\otimes$}
    ( 8,1)node[shape=circle,draw,fill=black,inner sep=0pt,minimum size=4pt,label=right:$p^{(k+1)}_{k+1}$] (nnbag42) {}
    ( 7,0)node[shape=circle,draw,fill=black,inner sep=0pt,minimum size=4pt,label=left:$p^{(k)}_{k}-\epsilon^{(k)}_{k,k+1}$] (nnbag43) {}
   
    (5,3)node[shape=circle,draw,label=below:$B^{(k+1)}_{k+2}$,inner sep=0] (nnbag5) {$\cup$}
     ( 4,3)node[shape=circle,draw,fill=black,inner sep=0pt,minimum size=4pt,label=left:$p^{(k+1)}_{k+2}$] (nnbag52) {}
    
    (7,4)node[shape=circle,draw,label=above:$B^{(k+1)}_{k+3}$,inner sep=0] (nnbag6) {$\otimes$}
    ( 8,4)node[shape=circle,draw,fill=black,inner sep=0pt,minimum size=4pt,label=right:$p^{(k+1)}_{k+3}$] (nnbag62) {};

      \draw[-](nnbag4)--(nnbag42);
      % \draw[-](nnbag4)--node[right] {$\epsilon^{(k+1)}_{k,k+1}$}(nnbag43);
     
      \draw[-](nnbag5)--(nnbag52);
    
      \draw[-](nnbag6)--(nnbag62);
      
      \draw[-](nnbag6)--node[right] {$\epsilon^{(k+1)}_{k+1,k+3}$}(nnbag4);
      
      \draw[-](nnbag6)--node[above,sloped] {$\epsilon^{(k+1)}_{k+2,k+3}$}(nnbag5);

       \draw [dotted] (nnbag6) -- (5.5,5);

    %%%%%%%%%%%%%%%%%%%%%%%%%%%%%%%%%%%

      \draw[->] (2,2) -- (2.5,2);
         
      \draw[-](nbag4)--(nbag42);
       \draw[-](nbag4)--node[right] {$\epsilon^{(k)}_{k,k+1}$}(nbag43);
     
      \draw[-](nbag5)--(nbag52);
    
      \draw[-](nbag6)--(nbag62);
      
      \draw[-](nbag6)--node[right] {$\epsilon^{(k)}_{k+1,k+3}$}(nbag4);
      
      \draw[-](nbag6)--node[above,sloped] {$\epsilon^{(k)}_{k+2,k+3}$}(nbag5);

       \draw [dotted] (nbag6) -- (-1.5,5);
    
%%%%%%%%%%%%%%%%%%%%%%%%%%%%%%%%%%%%%%%%%%%%%%%right

\end{tikzpicture}
\caption{\label{fig_lamma22} Vertex $v^{(k)}_{k}$ removed and the remaining threshold graph $T_{k+1}$.}
\end{figure}
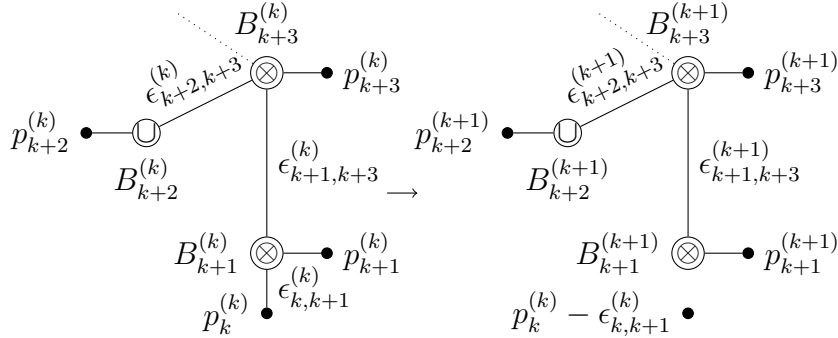

\end{proof}

\begin{Example}%[Continuation of Example~\ref{example_reduction}]

\label{example_smalltree}

In this example we will compute the spectrum of the weighted threshold graph $H'$ computed in Example \ref{example_reduction} and depicted in Figure \ref{secT1}. 

Using the notation defined in Lemma \ref{key}, we set $T_1=H'$, meaning that, $B_{i}^{(1)}=B_{i}'$ and $p_{i}^{(1)}=p_{i}'$ for $1\leq i\leq 6$, and $\epsilon_{i,j}^{(1)}=\epsilon_{i,j}'$ for $1\leq i<j\leq 6$, where $j$ is even. Therefore, according to Example  \ref{example_reduction}, the bag $B^{(1)}_{i}$ has a vertex $v^{(1)}_{i}$ with weight $p^{(1)}_{i}=0$ for $i\in\{1,2,5,6\}$ and $p^{(1)}_{i}=\lambda$ for $i\in\{3,4\}$.  
The weights of the edges between the bags are 
$\epsilon_{1,2}^{(1)}=\lambda$, $\epsilon_{1,4}^{(1)}=\frac{-\lambda}{2}$,  $\epsilon_{2,4}^{(1)}=\frac{-\lambda}{2}$,  $\epsilon_{3,4}^{(1)}=\frac{-\lambda\sqrt{2}}{2}$,  $\epsilon_{1,6}^{(1)}=\frac{\lambda}{4}$,  $\epsilon_{2,6}^{(1)}=\frac{\lambda}{4}$,  $\epsilon_{3,6}^{(1)}=\frac{\lambda\sqrt{2}}{4}$,  $\epsilon_{4,6}^{(1)}=\frac{\lambda}{2}$ and  $\epsilon_{5,6}^{(1)}=\frac{\lambda\sqrt{2}}{2}$.

Since the weights in the weighted threshold graph $T_{1}$ satisfy
\begin{equation}
\label{example_rest_1}
         \begin{aligned}
         %\label{eq32}
           p^{(1)}_{1}&=p^{(1)}_{2}&=0,~
          \epsilon^{(1)}_{1,4}=\epsilon^{(1)}_{2,4}=\frac{-\lambda}{2}, \textrm{ and }
          \epsilon^{(1)}_{1,6}=\epsilon^{(1)}_{2,6}=\frac{\lambda}{4},
         \end{aligned}
    \end{equation}
Lemma \ref{key} may be applied. Equation~\eqref{eq:spec} tells us that $\Spec(T_1)$ is the union of $\Spec(T_2)$, where $T_2$ is a smaller weighted threshold graph, and the value
$p^{(1)}_{1}-\epsilon^{(1)}_{1,2}=0-\lambda=-\lambda$. 

%(see Figures \ref{fig_example_1} and \ref{fig_example_2}).

%%%%%%%%%%%%%%%%%%%%%%%%%%% teste

%%%%%%%%%%%%%%%%%%%%%%%%%%%%%%%%%%%%%%%%%%%%%%%

The weighted threshold graph $T_{2}$ is depicted in Figure \ref{secT2}. By part (c) the weights of the edges incident with the bag $B^{(2)}_{2}$ are
    \begin{equation}
         \begin{aligned}
         %\label{eq32}
          \epsilon^{(2)}_{2,4}&=\sqrt{2}\epsilon^{(1)}_{2,4}&=\frac{-\lambda\sqrt{2}}{2} \textrm{ and }
          \epsilon^{(2)}_{2,6}=\sqrt{2}\epsilon^{(1)}_{2,6}=\frac{\lambda\sqrt{2}}{4}.
         \end{aligned}
    \end{equation}
    By part (d), the weights of the remaining edges remain unchanged, that is, $\epsilon^{(2)}_{1,6}=\epsilon^{(1)}_{1,6}$, $\epsilon^{(2)}_{3,6}=\epsilon^{(1)}_{3,6}$, $\epsilon^{(2)}_{4,6}=\epsilon^{(1)}_{4,6}$, $\epsilon^{(2)}_{5,6}=\epsilon^{(1)}_{5,6}$, $\epsilon^{(2)}_{1,4}=\epsilon^{(1)}_{1,4}$ and $\epsilon^{(2)}_{3,4}=\epsilon^{(1)}_{3,4}$. The weight $p^{(2)}_{2}$ is given by $p^{(2)}_{2}=p^{(1)}_{2}+\epsilon^{(1)}_{1,2}=0+\lambda=\lambda$ by part (b), while part (a) tells us that the remaining weights are the same: $p^{(2)}_{3}=p^{(1)}_{3}$, $p^{(2)}_{4}=p^{(1)}_{4}$, $p^{(2)}_{5}=p^{(1)}_{5}$ and $p^{(2)}_{6}=p^{(1)}_{6}$.

 %%%%%%%%%%%%%%%%%%%%%%%%%%%%%%%%%%%%%%%%%%%%%%%%%%%%%
 
 % \begin{figure}[htb]
 %    \begin{minipage}[t]{.45\textwidth}
 %        %\centering
 \vspace{-2cm}
\begin{figure}[H]
    \centering
    \begin{minipage}[t]{0.49\textwidth}
        \centering
        \hspace{-2.5cm}
        \begin{tikzpicture}[scale=1]
\path(0,0)node[shape=circle,draw,label=below:$B^{(1)}_{1}$,inner sep=0] (bag1) {$\cup$}
    ( -.5,0)node[shape=circle,draw,fill=black,inner sep=0pt,minimum size=4pt] (bag12) {}
    %( -.5,-0.5)node[shape=circle,draw,fill=black,inner sep=0pt,minimum size=4pt] (bag11) {}

    (2,1)node[shape=circle,draw,label=below:$B^{(1)}_{2}$,inner sep=0] (bag2) {$\otimes$}
    ( 2.5,1)node[shape=circle,draw,fill=black,inner sep=0pt,minimum size=4pt] (bag22) {}
    %( 2.5,0.5)node[shape=circle,draw,fill=black,inner sep=0pt,minimum size=4pt] (bag21) {}

    (0,2)node[shape=circle,draw,label=below:$B^{(1)}_{3}$,inner sep=0] (bag3) {$\cup$}
    ( -0.5,2)node[shape=circle,draw,fill=black,inner sep=0pt,minimum size=4pt] (bag32) {}
    %( -0.5,1.5)node[shape=circle,draw,fill=black,inner sep=0pt,minimum size=4pt] (bag31) {}

    (2,3)node[shape=circle,draw,label=below:$B^{(1)}_{4}$,inner sep=0] (bag4) {$\otimes$}
    ( 2.5,3)node[shape=circle,draw,fill=black,inner sep=0pt,minimum size=4pt] (bag42) {}
    %( 2.5,2.5)node[shape=circle,draw,fill=black,inner sep=0pt,minimum size=4pt] (bag41) {}

    (0,4)node[shape=circle,draw,label=below:$B^{(1)}_{5}$,inner sep=0] (bag5) {$\cup$}
    ( -0.5,4)node[shape=circle,draw,fill=black,inner sep=0pt,minimum size=4pt] (bag52) {}
%    ( -0.5,3.5)node[shape=circle,draw,fill=black,inner sep=0pt,minimum size=4pt] (bag51) {}

    (2,5)node[shape=circle,draw,label=below:$B^{(1)}_{6}$,inner sep=0] (bag6) {$\otimes$}
    ( 2.5,5)node[shape=circle,draw,fill=black,inner sep=0pt,minimum size=4pt] (bag62) {};
 %   ( 2.5,4.5)node[shape=circle,draw,fill=black,inner sep=0pt,minimum size=4pt] (bag61) {};
     
     %( 2,0)node[shape=circle,draw,label=below:$p_{k_{1}}$,fill=black] (pri2) {}
     %( 4,0)node[shape=circle,draw,label=right:$p_{k_{1}}$,fill=black] (seg1) {}  
     %( 4,1)node[shape=circle,draw,label=right:$p_{k_{1}}$,fill=black] (seg2) {}  
     %( 6,1)node[shape=circle,draw,inner sep=0] (ter) {$\cup$}
     %( 6,0)node[shape=circle,draw,label=right:$p_{k_{1}}$,fill=black] (ter1) {};
     
      %\draw[-](bag1)--(bag11);
      \draw[-](bag1)--(bag12);
      %\draw[dotted](bag11)--(bag12);

      %\draw[-](bag2)--(bag21);
      \draw[-](bag2)--(bag22);
      %\draw[dotted](bag21)--(bag22);

      %\draw[-](bag3)--(bag31);
      \draw[-](bag3)--(bag32);
      %\draw[dotted](bag31)--(bag32);

      %\draw[-](bag4)--(bag41);
      \draw[-](bag4)--(bag42);
      %\draw[dotted](bag41)--(bag42);

      %\draw[-](bag5)--(bag51);
      \draw[-](bag5)--(bag52);
      %\draw[dotted](bag51)--(bag52);

      %\draw[-](bag6)--(bag61);
      \draw[-](bag6)--(bag62);
      %\draw[dotted](bag61)--(bag62);

      \draw[-](bag2)--node[below,sloped] {$\epsilon^{(1)}_{1,2}$}(bag1);

      \draw[-](bag4)--node[above,sloped] {$\epsilon^{(1)}_{1,4}$}(bag1);
      \draw[-](bag4)--node[right] {$\epsilon^{(1)}_{2,4}$}(bag2);
      \draw[-](bag4)--node[above,sloped] {$\epsilon^{(1)}_{3,4}$}(bag3);

      \draw[-](bag6)--node[above,sloped] {$\epsilon^{(1)}_{3,6}$}(bag3);
      \draw[-](bag6)--node[right] {$\epsilon^{(1)}_{4,6}$}(bag4);
      \draw[-](bag6)--node[above,sloped] {$\epsilon^{(1)}_{5,6}$}(bag5);

      \draw[-] (bag6) .. controls +(up:3cm) and +(left:5cm) .. node[above,sloped] {$\epsilon^{(1)}_{1,6}$} (bag1);

       \draw[-] (bag6) .. controls +(up:3cm) and +(left:-3cm) .. node[above,sloped] {$\epsilon^{(1)}_{2,6}$} (bag2);
      
      %\draw[-](pri)--(pri2);
      %\draw[dotted](pri1)--(pri2);
      
      %\draw[dotted](seg1)--(seg2);
      
      %\draw[-](ter)--(ter1);
      
       %\node[text width=6cm, anchor=west, right] at (2.5,0.5) {$\rightarrow$};

\end{tikzpicture}
        \caption{\label{secT1} Weighted threshold graph $T_{1}$ with binary sequence $(010101)$.}
    \end{minipage}
    \hfill
    \begin{minipage}[t]{0.49\textwidth}
        \centering
        % \vspace{-8cm}
\begin{tikzpicture}[scale=1]
\path

%(0,0)node[shape=circle,draw,label=below:$B_{1}$,inner sep=0] (bag1) {$\cup$}
    %( 2,0)node[shape=circle,draw,fill=black,inner sep=0pt,minimum size=4pt,label=right:$v^{(1)}_{1}$] (bag12) {}
    %( -.5,-0.5)node[shape=circle,draw,fill=black,inner sep=0pt,minimum size=4pt] (bag11) {}

    (2,1)node[shape=circle,draw,inner sep=0] (bag2) {$\otimes$}
    ( 3,1)node[shape=circle,draw,fill=black,inner sep=0pt,minimum size=4pt,label=right:$\lambda$] (bag22) {}
    %( 2.5,0.5)node[shape=circle,draw,fill=black,inner sep=0pt,minimum size=4pt] (bag21) {}

    (0,2)node[shape=circle,draw,inner sep=0] (bag3) {$\cup$}
    ( -0.5,2)node[shape=circle,draw,fill=black,inner sep=0pt,minimum size=4pt,label=left:$\lambda$] (bag32) {}
    %( -0.5,1.5)node[shape=circle,draw,fill=black,inner sep=0pt,minimum size=4pt] (bag31) {}

    (2,3)node[shape=circle,draw,inner sep=0] (bag4) {$\otimes$}
    ( 2.5,3)node[shape=circle,draw,fill=black,inner sep=0pt,minimum size=4pt,label=right:$\lambda$] (bag42) {}
    %( 2.5,2.5)node[shape=circle,draw,fill=black,inner sep=0pt,minimum size=4pt] (bag41) {}

    (0,4)node[shape=circle,draw,inner sep=0] (bag5) {$\cup$}
    ( -0.5,4)node[shape=circle,draw,fill=black,inner sep=0pt,minimum size=4pt,label=left:$0$] (bag52) {}
%    ( -0.5,3.5)node[shape=circle,draw,fill=black,inner sep=0pt,minimum size=4pt] (bag51) {}

    (2,5)node[shape=circle,draw,inner sep=0] (bag6) {$\otimes$}
    ( 2.5,5)node[shape=circle,draw,fill=black,inner sep=0pt,minimum size=4pt,label=right:$0$] (bag62) {};
 %   ( 2.5,4.5)node[shape=circle,draw,fill=black,inner sep=0pt,minimum size=4pt] (bag61) {};
     
     %( 2,0)node[shape=circle,draw,label=below:$p_{k_{1}}$,fill=black] (pri2) {}
     %( 4,0)node[shape=circle,draw,label=right:$p_{k_{1}}$,fill=black] (seg1) {}  
     %( 4,1)node[shape=circle,draw,label=right:$p_{k_{1}}$,fill=black] (seg2) {}  
     %( 6,1)node[shape=circle,draw,inner sep=0] (ter) {$\cup$}
     %( 6,0)node[shape=circle,draw,label=right:$p_{k_{1}}$,fill=black] (ter1) {};
     
      %\draw[-](bag1)--(bag11);
      %\draw[-](bag1)--(bag12);
      %\draw[dotted](bag11)--(bag12);

      %\draw[-](bag2)--(bag21);
      \draw[-](bag2)--(bag22);
      %\draw[dotted](bag21)--(bag22);

      %\draw[-](bag3)--(bag31);
      \draw[-](bag3)--(bag32);
      %\draw[dotted](bag31)--(bag32);

      %\draw[-](bag4)--(bag41);
      \draw[-](bag4)--(bag42);
      %\draw[dotted](bag41)--(bag42);

      %\draw[-](bag5)--(bag51);
      \draw[-](bag5)--(bag52);
      %\draw[dotted](bag51)--(bag52);

      %\draw[-](bag6)--(bag61);
      \draw[-](bag6)--(bag62);
      %\draw[dotted](bag61)--(bag62);

     % \draw[-](bag2)--node[below,sloped] {}(bag12);

      %\draw[-](bag4)--node[above,sloped] {$\epsilon^{(1)}_{1,4}$}(bag1);
      \draw[-](bag4)--node[right] {$\frac{-\lambda\sqrt{2}}{2}$}(bag2);
      \draw[-](bag4)--node[below,sloped] {$\frac{-\lambda\sqrt{2}}{2}$}(bag3);

      \draw[-](bag6)--node[above,sloped] {$\frac{\lambda\sqrt{2}}{4}$}(bag3);
      \draw[-](bag6)--node[right] {$\frac{\lambda}{2}$}(bag4);
      \draw[-](bag6)--node[above,sloped] {$\frac{\lambda\sqrt{2}}{2}$}(bag5);

     % \draw[-,red] (bag6) .. controls +(up:3cm) and +(left:5cm) .. node[above,sloped] {$\epsilon^{(1)}_{1,6}$} (bag2);

       \draw[-] (bag6) .. controls +(up:3cm) and +(left:-3cm) .. node[above,sloped] {$\frac{\lambda\sqrt{2}}{4}$} (bag2);
      
      %\draw[-](pri)--(pri2);
      %\draw[dotted](pri1)--(pri2);
      
      %\draw[dotted](seg1)--(seg2);
      
      %\draw[-](ter)--(ter1);
      
       %\node[text width=6cm, anchor=west, right] at (2.5,0.5) {$\rightarrow$};

\end{tikzpicture}
        \caption{\label{secT2} Weighted threshold $T_2$ with $\Spec(T_1)=\Spec(T_2)\cup\{-\lambda\}$.}
    \end{minipage}
    % \caption{Comparison of two graph configurations.}
\end{figure}

Repeating this process, we obtain the sequence of weighted threshold graphs $T_3,T_4,T_5,T_6$ as depicted in Figures~\ref{secT3} to~\ref{secT6}, respectively.  

\begin{figure}[H]
    \centering
    \begin{minipage}[t]{0.49\textwidth}
        \centering
        \begin{tikzpicture}
            \path
            (0,2)node[shape=circle,draw,label=below:$B^{(3)}_{3}$,inner sep=0] (bag3) {$\cup$}
            (-0.5,2)node[shape=circle,draw,fill=black,inner sep=0pt,minimum size=4pt,label=left:$\lambda$] (bag32) {}
            
            (2,3)node[shape=circle,draw,label=below:$B^{(3)}_{4}$,inner sep=0] (bag4) {$\otimes$}
            (2.5,3)node[shape=circle,draw,fill=black,inner sep=0pt,minimum size=4pt,label=right:$\lambda$] (bag42) {}
            
            (0,4)node[shape=circle,draw,label=below:$B^{(3)}_{5}$,inner sep=0] (bag5) {$\cup$}
            (-0.5,4)node[shape=circle,draw,fill=black,inner sep=0pt,minimum size=4pt,label=left:$0$] (bag52) {}
            
            (2,5)node[shape=circle,draw,label=above:$B^{(3)}_{6}$,inner sep=0] (bag6) {$\otimes$}
            (2.5,5)node[shape=circle,draw,fill=black,inner sep=0pt,minimum size=4pt,label=right:$0$] (bag62) {};

            \draw[-](bag3)--(bag32);
            \draw[-](bag4)--(bag42);
            \draw[-](bag5)--(bag52);
            \draw[-](bag6)--(bag62);
            \draw[-](bag4)--node[below,sloped] {$-\lambda$}(bag3);
            \draw[-](bag6)--node[above,sloped] {$\frac{\lambda}{2}$}(bag3);
            \draw[-](bag6)--node[right] {$\frac{\lambda}{2}$}(bag4);
            \draw[-](bag6)--node[above,sloped] {$\frac{\lambda\sqrt{2}}{2}$}(bag5);
        \end{tikzpicture}
        \caption{\label{secT3} Weighted threshold $T_3$ with $\Spec(T_2)=\Spec(T_3)\cup\{\lambda\}$.}
    \end{minipage}
    \hfill
    \begin{minipage}[t]{0.49\textwidth}
        \centering
        % \vspace{-8cm}
        \begin{tikzpicture}
\path

%(0,0)node[shape=circle,draw,label=below:$B_{1}$,inner sep=0] (bag1) {$\cup$}
    %( 2,0)node[shape=circle,draw,fill=black,inner sep=0pt,minimum size=4pt,label=right:$v^{(1)}_{1}$] (bag12) {}
    %( -.5,-0.5)node[shape=circle,draw,fill=black,inner sep=0pt,minimum size=4pt] (bag11) {}

    %(2,1)node[shape=circle,draw,label=below:$B^{(2)}_{2}$,inner sep=0] (bag2) {$\otimes$}
    %( 0,1.5)node[shape=circle,draw,fill=black,inner sep=0pt,minimum size=4pt,label=right:$v^{(2)}_{2}$] (bag22) {}
    %( 2.5,0.5)node[shape=circle,draw,fill=black,inner sep=0pt,minimum size=4pt] (bag21) {}

    %(0,2)node[shape=circle,draw,label=below:$B^{(3)}_{3}$,inner sep=0] (bag3) {$\cup$}
    %( 2,2.5)node[shape=circle,draw,fill=black,inner sep=0pt,minimum size=4pt,label=below:$v^{(3)}_{3}$] (bag32) {}
    %( -0.5,1.5)node[shape=circle,draw,fill=black,inner sep=0pt,minimum size=4pt] (bag31) {}

    (2,3)node[shape=circle,draw,label=below:$B^{(4)}_{4}$,inner sep=0] (bag4) {$\otimes$}
    ( 2.5,3)node[shape=circle,draw,fill=black,inner sep=0pt,minimum size=4pt,label=right:$0$] (bag42) {}
    %( 2.5,2.5)node[shape=circle,draw,fill=black,inner sep=0pt,minimum size=4pt] (bag41) {}

    (0,4)node[shape=circle,draw,label=below:$B^{(4)}_{5}$,inner sep=0] (bag5) {$\cup$}
    ( -0.5,4)node[shape=circle,draw,fill=black,inner sep=0pt,minimum size=4pt,label=left:$0$] (bag52) {}
%    ( -0.5,3.5)node[shape=circle,draw,fill=black,inner sep=0pt,minimum size=4pt] (bag51) {}

    (2,5)node[shape=circle,draw,label=above:$B^{(4)}_{6}$,inner sep=0] (bag6) {$\otimes$}
    ( 2.5,5)node[shape=circle,draw,fill=black,inner sep=0pt,minimum size=4pt,label=right:$0$] (bag62) {};
 %   ( 2.5,4.5)node[shape=circle,draw,fill=black,inner sep=0pt,minimum size=4pt] (bag61) {};
     
     %( 2,0)node[shape=circle,draw,label=below:$p_{k_{1}}$,fill=black] (pri2) {}
     %( 4,0)node[shape=circle,draw,label=right:$p_{k_{1}}$,fill=black] (seg1) {}  
     %( 4,1)node[shape=circle,draw,label=right:$p_{k_{1}}$,fill=black] (seg2) {}  
     %( 6,1)node[shape=circle,draw,inner sep=0] (ter) {$\cup$}
     %( 6,0)node[shape=circle,draw,label=right:$p_{k_{1}}$,fill=black] (ter1) {};
     
      %\draw[-](bag1)--(bag11);
      %\draw[-](bag1)--(bag12);
      %\draw[dotted](bag11)--(bag12);

      %\draw[-](bag2)--(bag21);
     % \draw[-](bag3)--(bag22);
      %\draw[dotted](bag21)--(bag22);

      %\draw[-](bag3)--(bag31);
      %\draw[-](bag4)--(bag32);
      %\draw[dotted](bag31)--(bag32);

      %\draw[-](bag4)--(bag41);
      \draw[-](bag4)--(bag42);
      %\draw[dotted](bag41)--(bag42);

      %\draw[-](bag5)--(bag51);
      \draw[-](bag5)--(bag52);
      %\draw[dotted](bag51)--(bag52);

      %\draw[-](bag6)--(bag61);
      \draw[-](bag6)--(bag62);
      %\draw[dotted](bag61)--(bag62);

     % \draw[-](bag2)--node[below,sloped] {}(bag12);

      %\draw[-](bag4)--node[above,sloped] {$\epsilon^{(1)}_{1,4}$}(bag1);
      %\draw[-](bag4)--node[right] {$\epsilon^{(2)}_{2,4}$}(bag3);
     %\draw[-](bag4)--node[below,sloped] {$\epsilon^{(3)}_{3,4}$}(bag3);

      %\draw[-,red](bag6)--node[above,sloped] {$\epsilon^{(3)}_{3,6}$}(bag3);
      \draw[-](bag6)--node[right] {$\frac{\lambda\sqrt{2}}{2}$}(bag4);
      \draw[-](bag6)--node[above,sloped] {$\frac{\lambda\sqrt{2}}{2}$}(bag5);

     % \draw[-,red] (bag6) .. controls +(up:3cm)   +(left:5cm) .. node[above,sloped] {$\epsilon^{(1)}_{1,6}$} (bag2);

      % \draw[-,red] (bag6) .. controls +(up:3cm) and +(left:3cm) .. node[above,sloped] {$\epsilon^{(2)}_{2,6}$} (bag3);
      
      %\draw[-](pri)--(pri2);
      %\draw[dotted](pri1)--(pri2);
      
      %\draw[dotted](seg1)--(seg2);
      
      %\draw[-](ter)--(ter1);
      
       %\node[text width=6cm, anchor=west, right] at (2.5,0.5) {$\rightarrow$};

\end{tikzpicture}
        \caption{\label{secT4} Weighted threshold $T_4$ with $\Spec(T_3)=\Spec(T_4)\cup\{2\lambda\}$.}
    \end{minipage}
    % \caption{Comparison of two graph configurations.}
\end{figure}

\begin{figure}[H]
    \centering
    \begin{minipage}[t]{0.49\textwidth}
        \centering
        \begin{tikzpicture}
\path

%(0,0)node[shape=circle,draw,label=below:$B_{1}$,inner sep=0] (bag1) {$\cup$}
    %( 2,0)node[shape=circle,draw,fill=black,inner sep=0pt,minimum size=4pt,label=right:$v^{(1)}_{1}$] (bag12) {}
    %( -.5,-0.5)node[shape=circle,draw,fill=black,inner sep=0pt,minimum size=4pt] (bag11) {}

    %(2,1)node[shape=circle,draw,label=below:$B^{(2)}_{2}$,inner sep=0] (bag2) {$\otimes$}
    %( 0,1.5)node[shape=circle,draw,fill=black,inner sep=0pt,minimum size=4pt,label=right:$v^{(2)}_{2}$] (bag22) {}
    %( 2.5,0.5)node[shape=circle,draw,fill=black,inner sep=0pt,minimum size=4pt] (bag21) {}

    %(0,2)node[shape=circle,draw,label=below:$B^{(3)}_{3}$,inner sep=0] (bag3) {$\cup$}
    %( 2,2.5)node[shape=circle,draw,fill=black,inner sep=0pt,minimum size=4pt,label=below:$v^{(3)}_{3}$] (bag32) {}
    %( -0.5,1.5)node[shape=circle,draw,fill=black,inner sep=0pt,minimum size=4pt] (bag31) {}

    %(2,3)node[shape=circle,draw,label=left:$B^{(4)}_{4}$,inner sep=0] (bag4) {$\otimes$}
    %( 0,3.5)node[shape=circle,draw,fill=black,inner sep=0pt,minimum size=4pt,label=right:$v^{(4)}_{4}$] (bag42) {}
    %( 2.5,2.5)node[shape=circle,draw,fill=black,inner sep=0pt,minimum size=4pt] (bag41) {}

    (0,4)node[shape=circle,draw,label=above:$B^{(4)}_{5}$,inner sep=0] (bag5) {$\cup$}
    ( -0.5,4)node[shape=circle,draw,fill=black,inner sep=0pt,minimum size=4pt,label=left:$0$] (bag52) {}
%    ( -0.5,3.5)node[shape=circle,draw,fill=black,inner sep=0pt,minimum size=4pt] (bag51) {}

    (2,5)node[shape=circle,draw,label=above:$B^{(5)}_{6}$,inner sep=0] (bag6) {$\otimes$}
    ( 2.5,5)node[shape=circle,draw,fill=black,inner sep=0pt,minimum size=4pt,label=right:$0$] (bag62) {};
 %   ( 2.5,4.5)node[shape=circle,draw,fill=black,inner sep=0pt,minimum size=4pt] (bag61) {};
     
     %( 2,0)node[shape=circle,draw,label=below:$p_{k_{1}}$,fill=black] (pri2) {}
     %( 4,0)node[shape=circle,draw,label=right:$p_{k_{1}}$,fill=black] (seg1) {}  
     %( 4,1)node[shape=circle,draw,label=right:$p_{k_{1}}$,fill=black] (seg2) {}  
     %( 6,1)node[shape=circle,draw,inner sep=0] (ter) {$\cup$}
     %( 6,0)node[shape=circle,draw,label=right:$p_{k_{1}}$,fill=black] (ter1) {};
     
      %\draw[-](bag1)--(bag11);
      %\draw[-](bag1)--(bag12);
      %\draw[dotted](bag11)--(bag12);

      %\draw[-](bag2)--(bag21);
     % \draw[-](bag3)--(bag22);
      %\draw[dotted](bag21)--(bag22);

      %\draw[-](bag3)--(bag31);
      %\draw[-](bag4)--(bag32);
      %\draw[dotted](bag31)--(bag32);

      %\draw[-](bag4)--(bag41);
      %\draw[-](bag5)--(bag42);
      %\draw[dotted](bag41)--(bag42);

      %\draw[-](bag5)--(bag51);
      \draw[-](bag5)--(bag52);
      %\draw[dotted](bag51)--(bag52);

      %\draw[-](bag6)--(bag61);
      \draw[-](bag6)--(bag62);
      %\draw[dotted](bag61)--(bag62);

     % \draw[-](bag2)--node[below,sloped] {}(bag12);

      %\draw[-](bag4)--node[above,sloped] {$\epsilon^{(1)}_{1,4}$}(bag1);
      %\draw[-](bag4)--node[right] {$\epsilon^{(2)}_{2,4}$}(bag3);
     %\draw[-](bag4)--node[below,sloped] {$\epsilon^{(3)}_{3,4}$}(bag3);

      %\draw[-,red](bag6)--node[above,sloped] {$\epsilon^{(3)}_{3,6}$}(bag3);
      %\draw[-,red](bag6)--node[right] {$\epsilon^{(4)}_{4,6}$}(bag4);
      \draw[-](bag6)--node[above,sloped] {$\lambda$}(bag5);

     % \draw[-,red] (bag6) .. controls +(up:3cm) and +(left:5cm) .. node[above,sloped] {$\epsilon^{(1)}_{1,6}$} (bag2);

      % \draw[-,red] (bag6) .. controls +(up:3cm) and +(left:3cm) .. node[above,sloped] {$\epsilon^{(2)}_{2,6}$} (bag3);
      
      %\draw[-](pri)--(pri2);
      %\draw[dotted](pri1)--(pri2);
      
      %\draw[dotted](seg1)--(seg2);
      
      %\draw[-](ter)--(ter1);
      
       %\node[text width=6cm, anchor=west, right] at (2.5,0.5) {$\rightarrow$};

\end{tikzpicture}

        \caption{\label{secT5}  Weighted threshold $T_5$ with $\Spec(T_4)=\Spec(T_5)\cup\{0\}$.}
    \end{minipage}
    \hfill
    \begin{minipage}[t]{0.49\textwidth}
        \centering
        % \vspace{-8cm}
        \begin{tikzpicture}
\path

%(0,0)node[shape=circle,draw,label=below:$B_{1}$,inner sep=0] (bag1) {$\cup$}
    %( 2,0)node[shape=circle,draw,fill=black,inner sep=0pt,minimum size=4pt,label=right:$v^{(1)}_{1}$] (bag12) {}
    %( -.5,-0.5)node[shape=circle,draw,fill=black,inner sep=0pt,minimum size=4pt] (bag11) {}

    %(2,1)node[shape=circle,draw,label=below:$B^{(2)}_{2}$,inner sep=0] (bag2) {$\otimes$}
    %( 0,1.5)node[shape=circle,draw,fill=black,inner sep=0pt,minimum size=4pt,label=right:$v^{(2)}_{2}$] (bag22) {}
    %( 2.5,0.5)node[shape=circle,draw,fill=black,inner sep=0pt,minimum size=4pt] (bag21) {}

    %(0,2)node[shape=circle,draw,label=below:$B^{(3)}_{3}$,inner sep=0] (bag3) {$\cup$}
    %( 2,2.5)node[shape=circle,draw,fill=black,inner sep=0pt,minimum size=4pt,label=below:$v^{(3)}_{3}$] (bag32) {}
    %( -0.5,1.5)node[shape=circle,draw,fill=black,inner sep=0pt,minimum size=4pt] (bag31) {}

    %(2,3)node[shape=circle,draw,label=left:$B^{(4)}_{4}$,inner sep=0] (bag4) {$\otimes$}
    %( 0,3.5)node[shape=circle,draw,fill=black,inner sep=0pt,minimum size=4pt,label=right:$v^{(4)}_{4}$] (bag42) {}
    %( 2.5,2.5)node[shape=circle,draw,fill=black,inner sep=0pt,minimum size=4pt] (bag41) {}

    % (0,4)node[shape=circle,draw,label=above:$B^{(4)}_{5}$,inner sep=0] (bag5) {$\cup$}
    % ( -0.5,4)node[shape=circle,draw,fill=black,inner sep=0pt,minimum size=4pt,label=below:$v^{(5)}_{5}$] (bag52) {}
%    ( -0.5,3.5)node[shape=circle,draw,fill=black,inner sep=0pt,minimum size=4pt] (bag51) {}

    (2,5)node[shape=circle,draw,label=below:$B^{(6)}_{6}$,inner sep=0] (bag6) {$\otimes$}
    ( 2.5,5)node[shape=circle,draw,fill=black,inner sep=0pt,minimum size=4pt,label=right:$\lambda$] (bag62) {};
 %   ( 2.5,4.5)node[shape=circle,draw,fill=black,inner sep=0pt,minimum size=4pt] (bag61) {};
     
     %( 2,0)node[shape=circle,draw,label=below:$p_{k_{1}}$,fill=black] (pri2) {}
     %( 4,0)node[shape=circle,draw,label=right:$p_{k_{1}}$,fill=black] (seg1) {}  
     %( 4,1)node[shape=circle,draw,label=right:$p_{k_{1}}$,fill=black] (seg2) {}  
     %( 6,1)node[shape=circle,draw,inner sep=0] (ter) {$\cup$}
     %( 6,0)node[shape=circle,draw,label=right:$p_{k_{1}}$,fill=black] (ter1) {};
     
      %\draw[-](bag1)--(bag11);
      %\draw[-](bag1)--(bag12);
      %\draw[dotted](bag11)--(bag12);

      %\draw[-](bag2)--(bag21);
     % \draw[-](bag3)--(bag22);
      %\draw[dotted](bag21)--(bag22);

      %\draw[-](bag3)--(bag31);
      %\draw[-](bag4)--(bag32);
      %\draw[dotted](bag31)--(bag32);

      %\draw[-](bag4)--(bag41);
      %\draw[-](bag5)--(bag42);
      %\draw[dotted](bag41)--(bag42);

      %\draw[-](bag5)--(bag51);
      % \draw[-](bag5)--(bag52);
      %\draw[dotted](bag51)--(bag52);

      %\draw[-](bag6)--(bag61);
      \draw[-](bag6)--(bag62);
      %\draw[dotted](bag61)--(bag62);

     % \draw[-](bag2)--node[below,sloped] {}(bag12);

      %\draw[-](bag4)--node[above,sloped] {$\epsilon^{(1)}_{1,4}$}(bag1);
      %\draw[-](bag4)--node[right] {$\epsilon^{(2)}_{2,4}$}(bag3);
     %\draw[-](bag4)--node[below,sloped] {$\epsilon^{(3)}_{3,4}$}(bag3);

      %\draw[-,red](bag6)--node[above,sloped] {$\epsilon^{(3)}_{3,6}$}(bag3);
      %\draw[-,red](bag6)--node[right] {$\epsilon^{(4)}_{4,6}$}(bag4);
      % \draw[-](bag6)--node[above,sloped] {$\epsilon^{(5)}_{5,6}$}(bag5);

     % \draw[-,red] (bag6) .. controls +(up:3cm) and +(left:5cm) .. node[above,sloped] {$\epsilon^{(1)}_{1,6}$} (bag2);

      % \draw[-,red] (bag6) .. controls +(up:3cm) and +(left:3cm) .. node[above,sloped] {$\epsilon^{(2)}_{2,6}$} (bag3);
      
      %\draw[-](pri)--(pri2);
      %\draw[dotted](pri1)--(pri2);
      
      %\draw[dotted](seg1)--(seg2);
      
      %\draw[-](ter)--(ter1);
      
       %\node[text width=6cm, anchor=west, right] at (2.5,0.5) {$\rightarrow$};

\end{tikzpicture}

        \caption{\label{secT6}  Weighted threshold $T_6$ with $\Spec(T_5)=\Spec(T_6)\cup\{-\lambda\}$.}
    \end{minipage}
    % \caption{Comparison of two graph configurations.}
\end{figure}

Since $T_6$ is a singleton, $\Spec(T_6)=\{\lambda\}$. We conclude that
\begin{equation}\nonumber
         %\label{eq32}
          \mbox{Spec}(T_{1})=\{-\lambda^{(2)},\lambda,2\lambda,0\} \cup \Spec(T_6)=\{\lambda^{(2)},-\lambda^{(2)},0,2\lambda\}.
\end{equation}
%\begin{equation}\nonumber
 %        \begin{aligned}
         %\label{eq32}
 %         \mbox{Spec}(T_{6})&=\{\lambda\} \\
 %         \mbox{Spec}(T_{5})&=\mbox{Spec}(T_{6})\cup\{-\lambda\}=\{\lambda,-\lambda\}\\
 %         \mbox{Spec}(T_{4})&=\mbox{Spec}(T_{5})\cup\{0\}=\{\lambda,-\lambda,0\}\\
  %        \mbox{Spec}(T_{3})&=\mbox{Spec}(T_{4})\cup\{2\lambda\}=\{\lambda,-\lambda,0,2\lambda\}\\
   %       \mbox{Spec}(T_{2})&=\mbox{Spec}(T_{3})\cup\{\lambda\}=\{\lambda^{(2)},-\lambda,0,2\lambda\}\\
    %      \mbox{Spec}(T_{1})&=\mbox{Spec}(T_{2})\cup\{-\lambda\}=\{\lambda^{(2)},-\lambda^{(2)},0,2\lambda\}.
 %    \end{aligned}
%\end{equation}

Going back to Example \ref{example_reduction}, we see that the weighted threshold graph $G$ of Example \ref{example_reduction} has spectrum
\begin{equation}\nonumber
         \begin{aligned}
         %\label{eq32}
          \mbox{Spec}(G)&=\mbox{Spec}(T_{1})\cup\{0^{(a_{1}+a_{4}+a_{5}-3)},\lambda^{(a_{2}+a_{3}+a_{6}-3)}\}\\ &=\{0^{(a_{1}+a_{4}+a_{5}-2)},\lambda^{(a_{2}+a_{3}+a_{6}-1)},-\lambda^{(2)},2\lambda\}.
         \end{aligned}
    \end{equation}
\end{Example}
\color{black}

\section{Proof of Theorem \ref{main_theorem}}
\label{sec:cons}
%%%%%%%%%%%%%
The aim of this section is to provide a proof of our main result, Theorem \ref{main_theorem}, which states that, for any threshold graph $G$ and any nonzero real number $\lambda$, there exists a matrix $M \in S(G)$ such that $\DSpec(M) \subseteq \{-\lambda,0,\lambda,2\lambda\}$. We shall prove that this matrix $M$ may be defined as a uniform weighted threshold graph. To this end, let $G=G(p,\epsilon)$ be a uniform weighted threshold graph with binary representation given by $b= (0^{a_{1}}1^{a_{2}}0^{a_{3}}1^{a_{4}}\cdots b_{r}^{a_{r}})$ with $a_{i}\geq 1$ for $1\leq i\leq r$ and vertex weights defined by
\begin{equation}\nonumber
    \begin{cases}
         p_{i}=0 \mbox{ and } \epsilon_{i}=0 & \mbox{if } i \equiv 1 \,(\bmod~ 4),\\
         p_{i}=\frac{\lambda(a_{i}-1)}{a_{i}} \mbox{ and } \epsilon_{i}=-\frac{\lambda}{a_{i}} & \mbox{if } i \equiv 2 \,(\bmod~ 4),\\
        p_{i}=\lambda \mbox{ and } \epsilon_{i}=0 & \mbox{if } i \equiv 3 \,(\bmod~ 4), \\
         p_{i}=\frac{\lambda}{a_{i}} \mbox{ and } \epsilon_{i}=\frac{\lambda}{a_{i}} & \mbox{if } i \equiv 0 \,(\bmod~ 4),
    \end{cases}
\end{equation}
with $\lambda\neq 0$.
By Corollary \ref{reduced_graph}
we have that $\mbox{DSpec}(G)\subseteq\mbox{DSpec}(T_{1})\cup\{0,\lambda\}$ regardless of the choice of edge-weights $\epsilon_{i,j}$, where $T_1$ is a uniform weighted threshold graph with the same number of bags, but for which each bag contains a single vertex.
Moreover, the edge-weights $\epsilon_{i,j}$ of $G$ relate to the edge weights $\epsilon^{(1)}_{i,j}$ of $T_{1}$ by the equation 
\begin{equation}\label{eq:rel}\epsilon_{i,j}=\frac{\epsilon^{(1)}_{i,j}}{\sqrt{a_{i}a_{j}}}.
\end{equation}
To derive
Theorem \ref{main_theorem}, it suffices to show that it is possible to assign edge-weights to $T_{1}$ in a way that $\DSpec(T_1)\subseteq \{-\lambda,0,\lambda,2\lambda\}$. This is the subject of Lemma \ref{sequence_of_graphs} below. Because this result is more technical, we start with a brief informal description. Suppose that we started with a weighted threshold graph $T_1$ whose bag representation has $r$ bags and whose edge-weights are unknown nonzero values. If these edge-weights satisfy the constraints described in Lemma~\ref{key}, we can apply it to describe the spectrum of $T_1$ in terms of the spectrum of a uniform weighted threshold graph $T_2$ with one fewer bag. Again, if the edge-weights of $T_2$ satisfy the constraints of Lemma~\ref{key}, a new reduction step may be performed. 
% The idea is to ensure that the edge-weights can be chosen in a way that keeps the process going until we reach a singleton $T_r$. Unfortunately, this process cannot be described by fixing the weights one by one, but rather by defining a sequence of constraints that will allow us to compute all the weights at the end of the process.
The idea is to ensure that the edge-weights can be chosen in a way that keeps the process going until we reach a singleton $T_r$ with vertex weight 0 or $\lambda$. Unfortunately, this
process cannot be described by fixing the weights one by one, but rather by defining a
sequence of constraints that tie these weights to the value of weights that will be defined later in the process.

\begin{lemat}
\label{sequence_of_graphs}
Let $r$ be a positive integer and let $T_{1}$ be
a threshold graph with bag representation given by $\{B_{i}^{(1)}\}_{i=1}^{r}$, where $|B_{i}^{(1)}|=1$ for every $i$. Suppose that each bag $B_{i}^{(1)}$ is assigned weight
 $p^{(1)}_{i}=0$ if $i\equiv q \mod 4$ for $q\in\{1,2\}$ and $p^{(1)}_{i}=\lambda$ if $i\equiv q \mod 4$ for $q\in\{0,3\}$. Then the following hold:
\begin{enumerate}
    \item\label{itemi} If $r=1$, then $T_{1}=v_{1}^{(1)}$ with weight $p^{(1)}_{1}=0$ satisfies Spec$(T_{1})=\{0^{(1)}\}$.\label{itemi} 
    \item \label{itemii} If $r>1$, then the following holds for every $i\in\{1,\ldots,r-1\}$. For every choice of nonzero edge-weights $\epsilon_{\ell,j}^{(1)}$, where $i<\ell<j\leq r$, and $j$ is even, there exist nonzero edge-weights $\epsilon_{\ell,j}^{(1)}$ for all $\ell \leq i$ and all $j$ even with $j>\ell$, with the property that there exists a sequence of weighted threshold graphs $T_1(p^{(1)},\epsilon^{(1)}),T_{2}(p^{(2)},\epsilon^{(2)}),\ldots,T_{i+1}(p^{(i+1)},\epsilon^{(i+1)})$ that satisfies the following for every $k\in\{1,\ldots,i\}$:
\begin{itemize}
    \item[(i)]$T_{k+1}$ has bag representation $\{B_{s}^{(k+1)}\}_{s=k+1}^{r}$, where $|B_{s}^{(k+1)}|=1$ for every $s$;
    \item[(ii)] 
    $\epsilon_{k,j}^{(k)}= \epsilon_{k+1,j}^{(k)}\mbox{ for }k+1<j\leq r,\mbox{ with $j$ even}$;
    \item[(iii)]\label{itemiib} If $q\in\{0,2\}$ is such that $k+1\equiv q~(\bmod \, 4)$, then $\epsilon^{(k)}_{k,k+1} = (q-1)\lambda$; 
    \item[(iv)] \label{itemiic} For $k+1<z < j\leq r$, where $j$ is even,
\begin{equation}\nonumber
    \begin{aligned}
    (a)~ p^{(k+1)}_{k+1}&=
    p^{(k)}_{k}+\epsilon^{(k)}_{k,k+1}\text{\footnotemark}
    % \textrm{\footnote{If $k$ is even, $\epsilon^{(k)}_{k,k+1}=0$.}} 
    &\text{ and }\text{ }\text{ } &\epsilon^{(k+1)}_{k+1,j}=\sqrt{2}\epsilon^{(k)}_{k+1,j}, \\
    (b)~ p^{(k+1)}_{z}&=p^{(k)}_{z}=p^{(1)}_{z} &\text{ and }\text{  } \text{ }&\epsilon^{(k+1)}_{z,j}=\epsilon^{(k)}_{z,j}=\epsilon^{(1)}_{z,j}; 
    \end{aligned}
    \end{equation}
\item[(v)]\label{itemiid} $\displaystyle{\Spec(T_{k})=\Spec(T_{k+1})\cup\{x_{k}\},}$ where $x_k = p_{k}^{(k)}-\epsilon_{k,k+1}^{(k)}$ if $k$ is odd and $x_k=p_{k}^{(k)}$ if $k$ is even. Moreover, 
$$x_k=
\begin{cases}
0,& \textrm{ if }k \equiv 0~(\bmod \, 4)\\
-\lambda,& \textrm{ if }k \equiv 1~(\bmod \, 4)\\
\lambda,& \textrm{ if }k \equiv 2~(\bmod \, 4)\\
2\lambda,& \textrm{ if }k \equiv 3~(\bmod \, 4).
\end{cases}$$
    
\end{itemize}

\end{enumerate}

\end{lemat}
\footnotetext{If $k$ is even, $\epsilon^{(k)}_{k,k+1}=0$.}

Note that part $(2.iv)$ implies that $p^{(r)}_r\stackrel{(iv.a)}=p^{(r-1)}_{r-1}+\epsilon_{r-1,r}^{(r-1)}  \stackrel{(v)}= \left(x_{r-1}+\epsilon_{r-1,r}^{(r-1)}\right)+\epsilon_{r-1,r}^{(r-1)}$, where we are assuming that  $\epsilon_{r-1,r}^{(r-1)}=0$ when $r-1$ is even. Assume that $r \equiv q ~(\bmod \, 4)$. We conclude that $p^{(r)}_r=0$ if $q \in \{0,1\}$, and that $p^{(r)}_r=\lambda$ if $q \in \{2,3\}$. 
%\stackrel{(iv)}=p^{(1)}_{r}+\epsilon_{r-1,r}^{(r-1)}

Before proving the lemma, we would like to give some intuition about part (2). The indices $i\in\{1,\ldots, r-1\}$ represent steps in an iterative procedure whose final objective is to find explicit weights for $T_1$. At each step $i$, the procedure defines the weights of the edges that are currently being processed in terms of edge-weights of edges that have yet to be processed. This is why the statement of (2) claims that, for any future assignment of weights to the unprocessed edges (i.e., $\epsilon_{\ell,j}$ for $i<\ell<j\leq r$), we may define the remaining weights (i.e., $\epsilon_{\ell,j}$ for $\ell\leq i$ and $\ell<j\leq r$) so that items $(i),\ldots,(v)$ are satisfied. When the procedure reaches $i=r-1$, all edges have been processed, so that part (2) is not conditioned on the choice of any unprocessed values. In other words, at this point all the edge-weights of $T_1$ have been defined. 

\begin{proof}
For $r=1$, the result is trivial.

We now consider $r\geq 2$ and proceed by induction on $i$. Let $T_1$ be as in the statement of the Theorem.

For $i=1$, let $\epsilon_{\ell,j}^{(1)}$ be arbitrary edge-weights for
%assume that an opponent has given us arbitrary nonzero edge-weights $\epsilon_{\ell,j}^{(1)}$, for 
all pairs $(\ell,j)$ such that $2 \leq \ell<j\leq r$, and $j$ is even. We need to find edge-weights $\epsilon_{1,j}^{(1)}$, for all even values of $j$, such that items (i)-(v) hold. For simplicity, we view these quantities $\epsilon_{1,j}^{(1)}$ as indeterminates.

By hypothesis, we have
    \begin{equation}
    \label{vertB2}
    p_1^{(1)}=p_2^{(1)}=0.
    \end{equation}
In order to satisfy (iii), we fix $\epsilon_{1,2}^{(1)}=\lambda$. We now impose the following constraints on the weights of the edges connected to vertices in bags $B^{(1)}_{1}$ and $B^{(1)}_{2}$:
    \begin{equation}\label{bag:b2}
        \epsilon^{(1)}_{1,j}=\epsilon^{(1)}_{2,j} \mbox{ for }2< j\leq r,\mbox{ $j$ even}.
    \end{equation}
The equation~\eqref{bag:b2} defines $\epsilon_{1,j}^{(1)}$ for all even $j>2$, which proves (ii).

Note that~\eqref{vertB2} and~\eqref{bag:b2} imply that Lemma~\ref{key} may be applied, which produces a weighted threshold graph $T_2=T_2(p^{(2)},\epsilon^{(2)})$ with bag representation $\{B_z^{(1)}\}_{z=2}^r$ (this proves (i)) such that
\begin{itemize}
\item[(a)] $p_z^{(2)}=p_z^{(1)}$ for all $z \geq 3$; 

\item[(b)] $p_2^{(2)}=p_1^{(1)}+\epsilon_{1,2}^{(1)}$; 

\item[(c)] $\epsilon_{z,j}^{(2)}=\epsilon_{z,j}^{(1)}$, for all $z \geq 3$ and all even $j>i$;

\item[(d)] $\epsilon_{2,j}^{(2)}=\sqrt{2} \epsilon^{(1)}_{2,j}$, where $j>2$ and $j$ is even.
\end{itemize}

Note that (a), (b) and (c) ensure that the vertex-weights $p^{(2)}_z$ are defined for all $z \geq 2$ and the edge-weights $\epsilon_{z,j}^{(2)}$ are defined for all $z\geq 3$ in a way that satisfies  (iv). Set $\epsilon_{2,j}^{(2)}=\sqrt{2} \epsilon_{2,j}^{(1)}$ for all even $j>2$, so that (iv) is also satisfied in this case. By Lemma~\ref{key}, $\Spec(T_1)=\Spec(T_2) \cup \{p_1^{(1)}-\epsilon_{1,2}^{(1)}\}=\Spec(T_2) \cup \{-\lambda\}$, so that (v) is satisfied. 

Assume that the lemma holds for $i-1$, where $1 \leq i-1 <r-1$. We show that it holds for $i$. To this end, 
let $\epsilon_{\ell,j}^{(1)}$ be arbitrary edge-weights for
%assume that an opponent has given us arbitrary nonzero edge-weights $\epsilon_{\ell,j}^{(1)}$, for 
all pairs $(\ell,j)$ such that $i+1 \leq \ell<j\leq r$, and $j$ is even. We need to find edge-weights $\epsilon_{\ell,j}^{(1)}$, for all $\ell \leq i$ and all even values of $j>\ell$, such that items (i)-(v) hold. 

By induction, for any choice of nonzero values for $\epsilon_{i,j}^{(1)}$ for all even $j>i$,  there exist nonzero values $\epsilon_{\ell,j}^{(1)}$ for all $\ell<i$ and all even $j>i-1$ for which the following holds. There exists a sequence of weighted threshold graphs $T_1(p^{(1)},\epsilon^{(1)}),T_{2}(p^{(2)},\epsilon^{(2)}),\ldots,T_{i}(p^{(i)},\epsilon^{(i)})$ that satisfies (i)-(v) for every $k \in \{1,\ldots,i-1\}$.

Our aim is to choose $\epsilon_{i,j}^{(1)}$ appropriately for all even $j>i$ to extend the above sequence to $T_{i+1}(p^{(i+1)},\epsilon^{(i+1)})$. The argument varies according to the remainder of $i$ modulo 4.

\vspace{4pt}

\noindent \textbf{Case 1: $i\equiv 1~(\bmod \, 4)$.} Consider the weighted threshold graph $T_{i}(p^{(i)},\epsilon^{(i)})$. If we have 
\begin{equation}\label{eq:ind1}
p_i^{(i)}=p^{(i)}_{i+1} \textrm{ and }\epsilon_{i,j}^{(i)} = \epsilon_{i+1,j}^{(i)}, \textrm{ for all even }j>i+1,
\end{equation}
then Lemma~\ref{key} may be applied. The first part of~\eqref{eq:ind1} follows by induction:
\begin{eqnarray*}&&p_i^{(i)}\stackrel{(iv.a)}{=}p_{i-1}^{(i-1)}\stackrel{(v)}{=}x_{i-1}=0\\
&&p_{i+1}^{(i)}\stackrel{(iv.b)}{=}p_{i+1}^{(1)}=0.
\end{eqnarray*}
The final equation comes from the fact that $i+1 \equiv 2~(\bmod \, 4)$. Note that by induction $\epsilon_{i+1,j}^{(i)}=\epsilon_{i+1,j}^{(i-1)}=\epsilon_{i+1,j}^{(1)}$ and $\epsilon_{i,j}^{(i)}=\sqrt{2}\epsilon_{i,j}^{(i-1)}=\sqrt{2}\epsilon_{i,j}^{(1)}$. Hence, to satisfy the second equation of~\eqref{eq:ind1}, we fix the appropriate values for $\epsilon_{i,j}^{(1)}$, namely
$$\epsilon_{i,j}^{(1)}=\frac{\epsilon_{i+1,j}^{(1)}}{\sqrt{2}}$$
for all even $j>i+1$. Moreover, since $i+1$ is even, in order to satisfy (iii), we must fix 
$$\lambda=\epsilon_{i,i+1}^{(i)}\stackrel{(iv.a)}{=}\sqrt{2}\epsilon_{i,i+1}^{(i-1)}\stackrel{(iv.b)}{=}\sqrt{2}\epsilon_{i,i+1}^{(1)},$$  
that is, $\epsilon_{i,i+1}^{(1)}=\lambda/\sqrt{2}$. Note that at this point all weights are fixed. 

By Lemma~\ref{key} applied to $T_i$, there is a weighted threshold graph $T_{i+1}(p^{(i+1)},\epsilon^{(i+1)})$ with bag representation $\{B_s^{(i+1)}\}_{s=i+1}^r$ such that:
\begin{itemize}
\item[(a)] $p_z^{(i+1)}=p_z^{(i)}\stackrel{(iv.b)}{=}p_z^{(1)}$ for all $z > i+1$; 

\item[(b)] $p_{i+1}^{(i+1)}=p_i^{(i)}+\epsilon_{i,i+1}^{(i)}=0+\lambda=\lambda$; 

\item[(c)] $\epsilon_{z,j}^{(i+1)}=\epsilon_{z,j}^{(i)}\stackrel{(iv.b)}{=}\epsilon_{z,j}^{(1)}$, for all $z > i+1$ and all even $j>i$;

\item[(d)] $\epsilon_{i+1,j}^{(i+1)}=\sqrt{2} \epsilon^{(i)}_{i+1,j}\stackrel{(iv.a)}{=}\sqrt{2} \epsilon^{(1)}_{i+1,j}$, for all even $j>i+1$.
\end{itemize}

We consider the sequence $T_1(p^{(1)},\epsilon^{(1)}),T_{2}(p^{(2)},\epsilon^{(2)}),\ldots,T_{i+1}(p^{(i+1)},\epsilon^{(i+1)})$ and we need to verify that it satisfies (i)-(v) for all $k\in \{1,\ldots,i\}$. By induction, this is clearly true for $k<i$, so consider the case $k=i$.

Item (i) holds by the definition of $T_{i+1}$. Item (ii) holds by our choice of $\epsilon_{i,j}^{(1)}$ for even $j>i+1$, while (iii) holds by our choice of $\epsilon_{i,i+1}^{(1)}$. Item (iv) is an immediate consequence of (a)-(d) above. It remains to prove (v). By Lemma~\ref{key}, $\Spec(T_i)=\Spec(T_{i+1})\cup \{x_i\}$, where $x_i=p_i^{(i)}-\epsilon_{i,i+1}^{(i)}=0-\lambda=-\lambda$, which is the desired result.

\vspace{4pt}

\noindent \textbf{Case 2: $i\equiv 2~(\bmod \, 4)$.} As before, we consider the weighted threshold graph $T_{i}(p^{(i)},\epsilon^{(i)})$. If we have 
\begin{equation}\label{eq:ind2}
p_i^{(i)}=p^{(i)}_{i+1} \textrm{ and }\epsilon_{i,j}^{(i)} = \epsilon_{i+1,j}^{(i)}, \textrm{ for all even }j>i+1,
\end{equation}
then Lemma~\ref{key} may be applied. The first part follows by induction:
\begin{eqnarray*}&&p_i^{(i)}\stackrel{(iv.a)}{=}p_{i-1}^{(i-1)}+\epsilon_{i-1,i}^{(i-1)}\stackrel{(v)}{=}x_{i-1}+2 \epsilon_{i-1,i}^{(i-1)} \stackrel{(v),(iii)}{=}-\lambda+2\lambda=\lambda,\\
&&p_{i+1}^{(i)}\stackrel{(iv.b)}{=}p_{i+1}^{(1)}=\lambda.
\end{eqnarray*}
The final equality comes from the fact that $i+1 \equiv 3~(\bmod \, 4)$. For the second equation of~\eqref{eq:ind2}, it suffices to repeat Case 1 and choose the appropriate values
$$\epsilon_{i,j}^{(1)}=\frac{\epsilon_{i+1,j}^{(1)}}{\sqrt{2}} \textrm{ for all even } j>i+1.$$ 
Since $i+1$ is odd, note that at this point all weights are fixed. 

By Lemma~\ref{key} applied to $T_i$, there is a weighted threshold graph $T_{i+1}(p^{(i+1)},\epsilon^{(i+1)})$ with bag representation $\{B_s^{(i+1)}\}_{s=i+1}^r$ such that:
\begin{itemize}
\item[(a)] $p_z^{(i+1)}=p_z^{(i)}\stackrel{(iv.b)}{=}p_z^{(1)}$ for all $z > i+1$; 

\item[(b)] $p_{i+1}^{(i+1)}=p_i^{(i)}=\lambda$; 

\item[(c)] $\epsilon_{z,j}^{(i+1)}=\epsilon_{z,j}^{(i)}\stackrel{(iv.b)}{=}\epsilon_{z,j}^{(1)}$, for all $z > i+1$ and all even $j>i$;

\item[(d)] $\epsilon_{i+1,j}^{(i+1)}=\sqrt{2} \epsilon^{(i)}_{i+1,j}\stackrel{(iv.a)}{=}\sqrt{2} \epsilon^{(1)}_{i+1,j}$, for all even $j>i+1$.
\end{itemize}

We consider the sequence $T_1(p^{(1)},\epsilon^{(1)}),T_{2}(p^{(2)},\epsilon^{(2)}),\ldots,T_{i+1}(p^{(i+1)},\epsilon^{(i+1)})$ and we need to verify that it satisfies (i)-(v) for all $k\in \{1,\ldots,i\}$. By induction, this is clearly true for $k<i$, so consider the case $k=i$.

Item (i) holds by the definition of $T_{i+1}$. Item (ii) holds by our choice of $\epsilon_{i,j}^{(1)}$ for even $j>i+1$, while (iii) does not apply in this case. Item (iv) is an immediate consequence of (a)-(d) above. It remains to prove (v). By Lemma~\ref{key}, $\Spec(T_i)=\Spec(T_{i+1})\cup \{x_i\}$, where $x_i=p_i^{(i)}=\lambda$, which is the desired result.

\vspace{4pt}

\noindent \textbf{Case 3: $i\equiv 3~(\bmod \, 4)$.} As in the previous cases, we consider the weighted threshold graph $T_{i}(p^{(i)},\epsilon^{(i)})$.  If we have 
\begin{equation}\label{eq:ind3}
p_i^{(i)}=p^{(i)}_{i+1} \textrm{ and }\epsilon_{i,j}^{(i)} = \epsilon_{i+1,j}^{(i)}, \textrm{ for all even }j>i+1,
\end{equation}
then Lemma~\ref{key} may be applied. The first part follows by induction:
\begin{eqnarray*}&&p_i^{(i)}\stackrel{(iv.a)}{=}p_{i-1}^{(i-1)}\stackrel{(v)}{=}x_{i-1}=\lambda\\
&&p_{i+1}^{(i)}\stackrel{(iv.b)}{=}p_{i+1}^{(1)}=\lambda,
\end{eqnarray*}
where the final equation comes from the fact that $i+1 \equiv 0~(\bmod \, 4)$. For the second equation of~\eqref{eq:ind3}, note that by induction $\epsilon_{i+1,j}^{(i)}=\epsilon_{i+1,j}^{(i-1)}=\epsilon_{i+1,j}^{(1)}$ and $\epsilon_{i,j}^{(i)}=\sqrt{2}\epsilon_{i,j}^{(i-1)}=\sqrt{2}\epsilon_{i,j}^{(1)}$. So, we fix the appropriate values for $\epsilon_{i,j}^{(1)}$, namely
$$\epsilon_{i,j}^{(1)}=\frac{\epsilon_{i+1,j}^{(1)}}{\sqrt{2}}$$
for all even $j>i+1$. Moreover, since $i+1$ is even, in order to satisfy (iii), we must fix 
$$-\lambda=\epsilon_{i,i+1}^{(i)}\stackrel{(iv.a)}{=}\sqrt{2}\epsilon_{i,i+1}^{(i-1)}\stackrel{(iv.b)}{=}\sqrt{2}\epsilon_{i,i+1}^{(1)},$$  
that is, $\epsilon_{i,i+1}^{(1)}=-\lambda/\sqrt{2}$. Note that at this point all weights are fixed. 

By Lemma~\ref{key} applied to $T_i$, there is a weighted threshold graph $T_{i+1}(p^{(i+1)},\epsilon^{(i+1)})$ with bag representation $\{B_s^{(i+1)}\}_{s=i+1}^r$ such that:
\begin{itemize}
\item[(a)] $p_z^{(i+1)}=p_z^{(i)}\stackrel{(iv.b)}{=}p_z^{(1)}$ for all $z > i+1$; 

\item[(b)] $p_{i+1}^{(i+1)}=p_i^{(i)}+\epsilon_{i,i+1}^{(i)}=\lambda-\lambda=0$; 

\item[(c)] $\epsilon_{z,j}^{(i+1)}=\epsilon_{z,j}^{(i)}\stackrel{(iv.b)}{=}\epsilon_{z,j}^{(1)}$, for all $z > i+1$ and all even $j>i$;

\item[(d)] $\epsilon_{i+1,j}^{(i+1)}=\sqrt{2} \epsilon^{(i)}_{i+1,j}\stackrel{(iv.a)}{=}\sqrt{2} \epsilon^{(1)}_{i+1,j}$, for all even $j>i+1$.
\end{itemize}

We consider the sequence $T_1(p^{(1)},\epsilon^{(1)}),T_{2}(p^{(2)},\epsilon^{(2)}),\ldots,T_{i+1}(p^{(i+1)},\epsilon^{(i+1)})$ and we need to verify that it satisfies (i)-(v) for all $k\in \{1,\ldots,i\}$. By induction, this is clearly true for $k<i$, so consider the case $k=i$.

Item (i) holds by the definition of $T_{i+1}$. Item (ii) holds by our choice of $\epsilon_{i,j}^{(1)}$ for even $j>i+1$, while (iii) holds by our choice of $\epsilon_{i,i+1}^{(1)}$. Item (iv) is an immediate consequence of (a)-(d) above. It remains to prove (v). By Lemma~\ref{key}, $\Spec(T_i)=\Spec(T_{i+1})\cup \{x_i\}$, where $x_i=p_i^{(i)}-\epsilon_{i,i+1}^{(i)}=\lambda-(-\lambda)=2\lambda$, which is the desired result.

\vspace{4pt}

\vspace{4pt}

\noindent \textbf{Case 4: $i\equiv 0~(\bmod \, 4)$.} As before, we consider the weighted threshold graph $T_{i}(p^{(i)},\epsilon^{(i)})$. If we have 
\begin{equation}\label{eq:ind4}
p_i^{(i)}=p^{(i)}_{i+1} \textrm{ and }\epsilon_{i,j}^{(i)} = \epsilon_{i+1,j}^{(i)}, \textrm{ for all even }j>i+1,
\end{equation}
then Lemma~\ref{key} may be applied. The first part follows by induction:
\begin{eqnarray*}&&p_i^{(i)}\stackrel{(iv.a)}{=}p_{i-1}^{(i-1)}+\epsilon_{i-1,i}^{(i-1)}\stackrel{(v)}{=}x_{i-1}+2 \epsilon_{i-1,i}^{(i-1)} \stackrel{(v),(iii)}{=}2\lambda+2(-\lambda)=0,\\
&&p_{i+1}^{(i)}\stackrel{(iv.b)}{=}p_{i+1}^{(1)}=0,
\end{eqnarray*}
where the final equality comes from the fact that $i+1 \equiv 1~(\bmod \, 4)$. For the second equation of~\eqref{eq:ind4}, it suffices to repeat the previous cases and choose the appropriate values
$$\epsilon_{i,j}^{(1)}=\frac{\epsilon_{i+1,j}^{(1)}}{\sqrt{2}} \textrm{ for all even } j>i+1.$$ 
Since $i+1$ is odd, note that at this point all weights are fixed. 

By Lemma~\ref{key} applied to $T_i$, there is a weighted threshold graph $T_{i+1}(p^{(i+1)},\epsilon^{(i+1)})$ with bag representation $\{B_s^{(i+1)}\}_{s=i+1}^r$ such that:
\begin{itemize}
\item[(a)] $p_z^{(i+1)}=p_z^{(i)}\stackrel{(iv.b)}{=}p_z^{(1)}$ for all $z > i+1$; 

\item[(b)] $p_{i+1}^{(i+1)}=p_i^{(i)}=0$; 

\item[(c)] $\epsilon_{z,j}^{(i+1)}=\epsilon_{z,j}^{(i)}\stackrel{(iv.b)}{=}\epsilon_{z,j}^{(1)}$, for all $z > i+1$ and all even $j>i$;

\item[(d)] $\epsilon_{i+1,j}^{(i+1)}=\sqrt{2} \epsilon^{(i)}_{i+1,j}\stackrel{(iv.a)}{=}\sqrt{2} \epsilon^{(1)}_{i+1,j}$, for all even $j>i+1$.
\end{itemize}

We consider the sequence $T_1(p^{(1)},\epsilon^{(1)}),T_{2}(p^{(2)},\epsilon^{(2)}),\ldots,T_{i+1}(p^{(i+1)},\epsilon^{(i+1)})$ and we need to verify that it satisfies (i)-(v) for all $k\in \{1,\ldots,i\}$. By induction, this is clearly true for $k<i$, so consider the case $k=i$.

Item (i) holds by the definition of $T_{i+1}$. Item (ii) holds by our choice of $\epsilon_{i,j}^{(1)}$ for even $j>i+1$, while (iii) does not apply in this case. Item (iv) is an immediate consequence of (a)-(d) above. It remains to prove (v). By Lemma~\ref{key}, $\Spec(T_i)=\Spec(T_{i+1})\cup \{x_i\}$, where $x_i=p_i^{(i)}=0$, which is the desired result.

\end{proof}
\section{Algorithm}
\label{sec:alg}

The next lemma provides explicit formulas for the edge weights $\epsilon_{ij}^{(1)}$ for the weighted threshold graph $T_1(p^{(1)},\epsilon^{(1)})$, which leads to a constructive procedure for obtaining a matrix $M \in \mathcal{S}(G)$ with at most four distinct eigenvalues for every threshold graph $G$.

%In the first one we show that once we have defined a edge weight for $\epsilon^{(i)}_{i,i+1}$, for $i$ odd, in the weighted threshold graph $T_{i}$ then we will have the edge weight $\epsilon^{(i-\ell)}_{i-\ell,i+1}$ in the weighted threshold graph $T_{i-\ell}$ connecting the bags $B^{(i-\ell)}_{i-\ell}$ and $B^{(i-\ell)}_{i+1}$.

% \begin{lemat}
% \label{rec}
%     For $i$ odd, if $\epsilon^{(i)}_{i,i+1}=c$ the $\epsilon^{(i-k)}_{i-k,i+1}=\frac{c}{\sqrt{2}^{k}}$ for $1\leq k\leq i-1$.
% \end{lemat}
%%%%%%%%%%%%%%%
\begin{lemat}
\label{rec}
Let $\lambda$ be a nonzero real number, and $T_1$ be
a threshold graph with bag representation given by $\{B_{i}^{(1)}\}_{i=1}^{r}$, where $|B_{i}^{(1)}|=1$ for every $i$ and $r>1$. Let $T_{1}(p^{(1)},\epsilon^{(1)}),T_{2}(p^{(2)},\epsilon^{(2)}),\ldots,T_{r}(p^{(r)},\epsilon^{(r)})$ be the sequence of weighted threshold graphs defined in Lemma \ref{sequence_of_graphs} with $\epsilon^{(1)}_{1,j}=\epsilon^{(1)}_{2,j}$ for even $4\leq j\leq r$.
% assume that $i+1\equiv q~(\bmod \, 4)$.
\begin{itemize}

\item[(a)] If $i$ is odd, let $q_i\in\{0,2\}$ be such that $i+1\equiv q_i~(\bmod \, 4)$. For $\ell\in \{0,\ldots, i-1\}$, we have
\begin{equation}
\label{edges0}\epsilon^{(i-\ell)}_{i-\ell,i+1}=\frac{(q_i-1)\lambda}{\sqrt{2}^{\ell}}.
\end{equation}

\item[(b)] 
%For all pairs $(i,j)$, where $2\leq i\leq r-1$ and 
If $j$ is even, let $q_j\in\{0,2\}$ be such that $j\equiv q_j~(\bmod \, 4)$. For $2 \leq i < j\leq r$, we have
    \begin{equation}
        \label{edges1}
        \epsilon^{(1)}_{i,j}=\frac{(q_j-1)\lambda}{\sqrt{2}^{j-i}}.
    \end{equation}

\end{itemize}

\end{lemat}

\begin{proof}

For $1\leq i\leq r-1$, let $q_i\in\{0,1,2,3\}$ be such that $i+1\equiv q_i~(\bmod \, 4)$. For part (a), we proceed our proof by induction on $\ell$. For $\ell=0$, we use the fact that $\epsilon^{(i)}_{i,i+1}=(q_i-1)\lambda$, which holds by part (iii) of Lemma~\ref{sequence_of_graphs}. Assume that the result holds for $\ell-1$, where $0 \leq \ell-1 <r-1$. Note that
\begin{eqnarray*}
\epsilon^{(i-\ell)}_{i-\ell,i+1}&\stackrel{(ii)}{=}&\epsilon^{(i-\ell)}_{i-\ell+1,i+1}\stackrel{(iv.a)}{=}\frac{\epsilon^{(i-\ell+1)}_{i-\ell+1,i+1}}{\sqrt{2}} \stackrel{(IH)}{=}\frac{1}{\sqrt{2}}\cdot\frac{(q_i-1)\lambda}{\sqrt{2}^{\ell-1}}=\frac{(q_i-1)\lambda}{\sqrt{2}^{\ell}},
\end{eqnarray*}
as required. Note that we justified the validity of each equation based on an item of Lemma~\ref{sequence_of_graphs} or on the induction hypothesis (IH).

For part (b), consider a pair $(i,j)$, where $2\leq i\leq r-1$, and $i<j\leq r$, where $j$ is even. Let $q_j\in\{0,2\}$ be such that $j\equiv q_j~(\bmod \, 4)$. We have
\begin{eqnarray*}
\epsilon^{(1)}_{i,j}&\stackrel{(iv.b)}{=}&\epsilon^{(i-1)}_{i,j}\stackrel{(iv.a)}{=}\frac{\epsilon^{(i)}_{i,j}}{\sqrt{2}}=\frac{\epsilon^{(j-1-(j-i-1))}_{(j-1-(j-i-1)),j}}{\sqrt{2}}\stackrel{(a)}{=}\frac{1}{\sqrt{2}}\cdot\frac{(q_j-1)\lambda}{\sqrt{2}^{j-i-1}}=\frac{(q_j-1)\lambda}{\sqrt{2}^{j-i}}.
\end{eqnarray*}

\end{proof}
%%%%%%%%%%%%%%%%%%%%%%%%%%%%%%%%%% using Lemma 5.2

 Lemma~\ref{rec} leads to an explicit algorithm for obtaining a matrix $M \in \mathcal{S}(G)$ with $\DSpec(M) \subseteq \{-\lambda,0,\lambda,2\lambda\}$ for every threshold graph $G$ and for every nonzero real number $\lambda$.

%The algorithm combines the results in Lemma~\ref{reduced_graph} and Lemma~\ref{sequence_of_graphs}.

%%%%%%%%%%%%%%%%%%%%%%%%%%%
\begin{algorithm}[H]
\caption{Generating a matrix $M\in S(G)$ such that DSpec$(M)\subseteq\{-\lambda,0,\lambda,2\lambda\}$.}\label{algT1}
\begin{algorithmic}
\Require a nonzero real number $\lambda$ and a sequence of positive integers $a_{1},\ldots,a_{r}$.
\Ensure a matrix $M\in S(G)$ such that DSpec$(M)\subseteq\{-\lambda,0,\lambda,2\lambda\}$.\\
Computing the weights of the vertices and internal edges:
\IIf{ $i \equiv 1 \,(\bmod~ 4)$} $p_{i}=0$ and $\epsilon_{i}=0$\EndIIf
    \IIf{ $i \equiv 2 \,(\bmod~ 4)$}  $p_{i}=\frac{\lambda(a_{i}-1)}{a_{i}}$ and $\epsilon_{i}=-\frac{\lambda}{a_{i}}$\EndIIf
    \IIf{ $i \equiv 3 \,(\bmod~ 4)$} $p_{i}=\lambda$ and $\epsilon_{i}=0$\EndIIf
    \IIf{ $i \equiv 0 \,(\bmod~ 4)$} $p_{i}=\frac{\lambda}{a_{i}}$ and $\epsilon_{i}=\frac{\lambda}{a_{i}}$\EndIIf
\\
Computing the weights of the edges between bags:
\For{$2\leq i\leq r-1$}
\For{$i< j\leq r$, $j$ even}
\IIf{$j \equiv 2 \,(\bmod~ 4)$ }
     $\epsilon^{(1)}_{i,j}=\frac{\lambda}{\sqrt{2}^{j-i}}$
\EndIIf
\IIf{$j \equiv 0 \,(\bmod~ 4)$}
    $\epsilon^{(1)}_{i,j}=\frac{-\lambda}{\sqrt{2}^{j-i}}$
\EndIIf
\EndFor
\EndFor\\
$\epsilon^{(1)}_{1,2}=\lambda$
\FFor{$4\leq j\leq r$, $j$ even}
 $\epsilon^{(1)}_{1,j}=\epsilon^{(1)}_{2,j}$
\EndFFor
% {\color{red} esse sao os pesos dos vertices da $T_{1}$, podemos tirar!}
% \For{$1\leq i\leq r$ a natural number}
% \IIf{$i\mbox{ mod $4$}\in\{1,2\}$ }
%      $\epsilon^{(1)}_{i,i}={\color{red}p^{(1)}_{i}}=0$
% \EndIIf
% % \IIf{$i \mod 4 = 2$ }
% %      $\epsilon^{(1)}_{i,i}=0$
% % \EndIIf
% \IIf{$i\mbox{ mod $4$}\in\{0,3\}$ }
%      $\epsilon^{(1)}_{i,i}={\color{red}p^{(1)}_{i}}=\lambda$
% \EndIIf
% % \IIf{$i \mod 4 = 0$ }
% %      $\epsilon^{(1)}_{i,i}=\lambda$
% % \EndIIf
% \EndFor
\For{ $1\leq i\leq r-1$}
\FFor{$i< j\leq r$, $j$ even}
 $\epsilon_{i,j}=\frac{\epsilon^{(1)}_{i,j}}{\sqrt{a_{i}a_{j}}}$
\EndFFor
\EndFor
% \If{$r$ is odd}
%     \State Choose $a_{r}\geq 1+\lfloor N+1\rfloor$
%     \State $s_{r} \gets1-\frac{1+N}{a_{r}}$  %\Comment{This is a comment}
% \ElsIf{$r$ is even}
%     \State  Choose $a_{r} \geq 2$
%     \State $s_{r} \gets  -\frac{N}{a_{r}}$
% \EndIf
% \For{$i=r-1$ to $1$}

% \If{$i$ is odd}
%     \State Choose $a_{i}\geq 1+\left\lfloor \frac{N+1}{1-\left(\frac{1}{s_{i+1}}\right)}\right\rfloor$
%     \State $p_{i} \gets s_{i+1}-1-\left(\frac{N+1}{a_{i}}\right) $ 
%     \State $s_{i} \gets f\left(s_{i+1},1-\left(\frac{N+1}{a_{i}}\right)\right)$
% \ElsIf{$i$ is even}
%     \State Choose $a_{i}\geq 1+\left\lfloor \frac{N}{s_{i+1}} \right\rfloor$
%     \State $p_{i} \gets s_{i+1}-\frac{N}{a_{i}} $ 
%     \State $s_{i} \gets g\left(s_{i+1},-\frac{N}{a_{i}}\right)$
% \EndIf

% \EndFor

\end{algorithmic}
\label{Alg_generating}
\end{algorithm}
%%%%%%%%%%%%%%%%%%%%%%%%%%%%%%%
% \begin{teore}
%     Algorithm \ref{Alg_generating} works.
% \end{teore}

% \begin{proof}
%     Just use Lemmas \ref{edgesnew},\ref{edges} and \ref{reduced_graph}.
% \end{proof}

% \begin{teore}
% \label{teo:spectrum}
%     Given a threshold graph $T\cong (0^{a_{1}}1^{a_{2}}0^{a_{3}}1^{a_{4}}\cdots t_{r}^{a_{r}})$, with $a_{i}\geq 1$ and $t_{r}=0$ if $r$ is odd and  $t_{r}=1$ if $r$ is even. Then there exists a matrix $M\in S(T)$ such that DSpec$(M)\subseteq \{-\lambda,0,\lambda,2\lambda\}$ for any nonzero real number $\lambda$.
% \end{teore}

% \begin{proof}
%     Just use Algorithm \ref{Alg_generating}.
% \end{proof}

Let $G$ be a threshold graph with binary sequence $(0^{a_{1}}1^{a_{2}}0^{a_{3}}1^{a_{4}}\cdots b_{r}^{a_{r}})$. Define $$A= \sum_{\substack{ 1 \leq i \leq n \\ i \equiv 1 \,(\bmod 4) }} (a_{i}-1)+\sum_{\substack{ 1 \leq i \leq n \\ i \equiv 0 \,(\bmod 4)  }} (a_{i}-1)\text{ and }B=\sum_{\substack{ 1 \leq i \leq n \\ i \equiv 2 \,(\bmod 4)  }} (a_{i}-1)+\sum_{\substack{ 1 \leq i \leq n \\ i \equiv 3 \,(\bmod 4)  }} (a_{i}-1).$$ We conclude this section by noting that, as a consequence of Lemmas \ref{reduced_graph} and \ref{sequence_of_graphs}, the matrix \( M \in S(G) \) generated by Algorithm 1 has the following spectrum for \( r = 4k + q \), where \( k \geq 0 \) is an integer and \( q \in \{0, 1, 2, 3\} \):
\begin{enumerate}
    \item [(a)] if $q=0$ then $\Spec(M)=\{-\lambda^{[k]},0^{[A+k]},\lambda^{[B+k]},2\lambda^{[k]}\}$,
    \item [(b)] if $q=1$ then $\Spec(M)=\{-\lambda^{[k]},0^{[A+k+1]},\lambda^{[B+k]},2\lambda^{[k]}\}$,
    \item [(c)] if $q=2$ then $\Spec(M)=\{-\lambda^{[k+1]},0^{[A+k]},\lambda^{[B+k+1]},2\lambda^{[k]}\}$,
    \item [(d)] if $q=3$ then $\Spec(M)=\{-\lambda^{[k+1]},0^{[A+k]},\lambda^{[B+k+2]},2\lambda^{[k]}\}$.
\end{enumerate}

\section{Final Remarks}
\label{sec:concl}

Fix a real number $\lambda \neq 0$ and a threshold graph $G$. In this paper, we have devised an iterative procedure that generates a matrix $M \in \mathcal{S}(G)$ with $\DSpec(M) \subseteq \{-\lambda,0,\lambda,2\lambda\}$. As a consequence, we have $q(G)\leq 4$ for every threshold graph $G$. This is an improvement over the previous known bound, for which $q(G)$ was bounded above by a linear function on the number of vertices of $G$. This is an important step towards classifying threshold graphs with respect to $q(G)$. Given a threshold graph $G$ with at least one edge, it remains to decide whether $q(G)=2$, $q(G)=3$ or $q(G)=4$. Recall that the structure of threshold graphs such that $q(G)=2$ has been studied in~\cite{fallat2022minimum}. 

A second research direction motivated by this paper is to find natural graph classes (in the library~\cite{isgci}, for example) for which the function $q(G)$ is bounded above by an absolute constant. We believe that the ideas in this paper can be extended to other classes, such as cographs.

A third research direction would be to identify graph classes $\mathcal{C}$ with the following property, which we call Property C. If $G \in \mathcal{C}$ is a graph with connected components $G_1,\ldots,G_s \in \mathcal{C}$, then $q(G) = \max_i q(G_i)$. It is easy to see that $q(G) \geq \max_i q(G_i)$. Corollary~\ref{union} implies that Property C may hold for threshold graphs. Indeed, it holds for threshold graphs $G$ such that $q(G)=4$ or $q(G)=1$. On the other hand, forests are known not to satisfy Property C, see \cite[Corollary 3.5]{AHSarxiv}.

% We conclude this paper addressing some open problems on the minimum number of distinct eigenvalues of threshold graphs. Our results show that $q(T)\leq 4$ for any threshold graph, so we have a constant upper bound for any threshold graph. This is an improvement over the previous known bound, which was linear on the number of vertices of the graph. Additionally, we provided a constructive method to obtain the matrix that achieves this bound. 
% neste paper apresentamso um metodo iterativo que como resultado q(t) <= 4. Esse improvement abre espaço para caracterizar todos os grafos threshold alem disso acreditamos que esse resultado pode ser extendido para cografos.
% % As a consequence, we proved that the maximum value of $q(G)$ for fixed $n$ is $4$ for threshold graphs.

% Looking forward on this topic, there are problems that remains open, such as:
% \begin{enumerate}
%     \item[(i)] Characterize which thresholds have $q(G)=3$.
%     \item[(ii)] Extend the results for other close classes of graphs, for instace cographs.
% \end{enumerate}

\section*{Acknowledgments}
We are thankful to two anonymous referees, whose careful reading and thoughtful comments improved the paper. Allem, Hoppen and Sibemberg acknowledge the partial support by
CAPES under project MATH-AMSUD 88881.694479/2022-01. They also acknowledge partial support by CNPq (Proj.\ 408180/2023-4). L.\ E.\ Allem was partially supported by FAPERGS 21/2551-
0002053-9. C.~Hoppen was partially supported of CNPq (Proj.\ 315132/2021-3). Tura was partially supported by CNPq Grant 200716/2022-0. CAPES is Coordena\c{c}\~{a}o de Aperfei\c{c}oamento de Pessoal de N\'{i}vel Superior. CNPq is Conselho Nacional de Desenvolvimento Cient\'{i}fico e Tecnol\'{o}gico. FAPERGS is Funda\c{c}\~{a}o de Amparo \`{a} Pesquisa do Estado do Rio Grande do Sul. 

\bibliographystyle{amsplain}
\input{main.bbl}

\end{document}

%% file: main.bbl
\providecommand{\bysame}{\leavevmode\hbox to3em{\hrulefill}\thinspace}
\providecommand{\MR}{\relax\ifhmode\unskip\space\fi MR }
\providecommand{\MRhref}[2]{%
  \href{http://www.ams.org/mathscinet-getitem?mr=#1}{#2}
}
\providecommand{\href}[2]{#2}